\documentclass[reqno, 11pt]{amsart}

\usepackage[english]{babel}
\usepackage{amsmath}
\usepackage{amsthm}
\usepackage{amsfonts}
\usepackage{amssymb}
\usepackage{anysize}
\usepackage{hyperref}
\usepackage{appendix}
\usepackage{mathtools}
\usepackage{graphicx}
\usepackage{xcolor}

\newtheorem{defi}{Definition}[section] 
\newtheorem{lem}[defi]{Lemma}
\newtheorem{theo}[defi]{Theorem}
\newtheorem{cor}[defi]{Corollary}
\newtheorem{pro}[defi]{Proposition}

\newtheorem{rem}[defi]{Remark}

\DeclareMathOperator{\N}{\mathbb{N}}
\DeclareMathOperator{\R}{\mathbb{R}}
\DeclareMathOperator{\C}{\mathbb{C}}
\DeclareMathOperator{\T}{\mathbb{T}}
\DeclareMathOperator{\D}{\mathbb{D}}
\DeclareMathOperator{\Z}{\mathbb{Z}}

\allowdisplaybreaks 

\title[K\'arm\'an Vortex Street in incompressible fluid models]{K\'arm\'an Vortex Street in incompressible fluid models}

\author[C. Garc\'ia]{Claudia Garc\'ia}
\address{Departamento de Matem\'atica Aplicada and Research Unit ``Modeling Nature'' (MNat), Facultad de Ciencias. Universidad de Granada\\ 18071-Granada, Spain. \& IRMAR, Universit\'e de Rennes 1\\ Campus de
Beaulieu\\ 35~042 Rennes cedex\\ France}
\email{claudiagarcia@ugr.es}
{\thanks{This work has been partially supported by the MINECO-Feder (Spain) research grant number MTM2017-91054EXP, the Junta de Andaluc\'ia (Spain) Project FQM 954 and the MECD (Spain) research grant FPU15/04094.}}


\begin{document}

\date{\today}

\begin{abstract}
This paper aims to provide a robust model for the well--known phenomenon of K\'arm\'an Vortex Street arising in Fluid Mechanics. The first theoretical attempt to model this pattern was given by von K\'arm\'an \cite{Karman1,Karman2}. He considered two parallel staggered rows of point vortices, with opposite strength, that translate at the same speed. Following the ideas of Saffman and Schatzman \cite{Saffman-Schatzman}, we propose to study this phenomenon in the Euler equations by considering two infinite arrows of vortex patches. The key idea is to  desingularize the point vortex model proposed by von K\'arm\'an. Our construction is flexible and can be extended to more general incompressible models.
\end{abstract}

\maketitle

\tableofcontents

\section{Introduction}

A distinctive pattern is observed in the wake of a two--dimensional bluff body placed in a uniform stream at certain velocities: this is called the von K\'arm\'an Vortex Street. It consists of vortices of high vorticity in an irrotational fluid. This complex phenomenon occurs in a large amount of circumstances. For instance, the singing of power transmission lines or kite strings in the wind. It can also be observed in atmospheric flow about islands, in ocean flows about pipes or in aeronautical systems. See \cite{Karman1,Karman2,Rayleigh, Saffman-Schatzman, Saffman-Schatzman2, Schatzman, Strouhal}.

The phenomenon of periodic vortex shedding, which is an oscillating flow that takes place when a fluid past a bluff body,  was first studied experimentally with the works \cite{Rayleigh, Strouhal}. The more classical model for K\'arm\'an Vortex Street, and the first theoretical one, was presented by von K\'arm\'an in \cite{Karman1, Karman2}. He considered a distribution of point vortices located in two parallel staggered rows, where the strength is opposite in each row. Other works concerning this model are \cite{Aref, Lamb, Rosenhead}. Since the exact problem seems to be complex from a theoretical point of view, there has been a large variety of approximations. 

It is observed that viscosity is involved with the generation of the vortex layers by the body, which generates the shedding process. However, it seems that viscosity does not play an essential role in the formation of the vortex street once the vortex layers have been created. Moreover, it would appear also that the body is not important in the formation and the evolution of the vortex street. If we concentrate on the evolution of the street, and not in the vortex layer formation process, an inviscid incompressible fluid model can be proposed to model this situation.

K\'arm\'an Vortex Street structures arise in the context of the Euler equations in the works of Saffman and Schatzman \cite{Saffman-Schatzman, Saffman-Schatzman2, Schatzman}. Following the ideas of von K\'arm\'an \cite{Karman1, Karman2}, they considered two parallel infinite array of vortices with finite area and uniform vorticity. They found numerically the existence of this kind of solutions that translate at a constant speed and studied their linear stability.

In this paper, we focus on the study of these structures in different inviscid incompressible fluid models via a desingularization of the point vortex model proposed by von K\'arm\'an \cite{Karman1, Karman2}. We obtain two infinite arrows of vortex patches, i.e. vortices with finite area and uniform vorticity, that translate. Then, we recover analytically the solutions found numerically by Saffman and Schatzman, not only in the Euler equations framework, but also in other incompressible models, that we recall in the following.

Let $q$ be a scalar magnitude of the fluid satisfying the following transport equation 
\begin{eqnarray}\label{Generalsystem}
       \left\{\begin{array}{ll}
          	\partial_t q+(v\cdot \nabla) q=0, &\text{ in $[0,+\infty)\times\mathbb{R}^2$}, \\
         	 v=\nabla^\perp \psi,&\text{ in $[0,+\infty)\times\mathbb{R}^2$}, \\
         	 \psi=G*q,&\text{ in $[0,+\infty)\times\mathbb{R}^2$}, \\
         	 q(0,x)=q_0(x),& \text{ with $x\in\mathbb{R}^2$}.
       \end{array}\right.
\end{eqnarray}
Note that the velocity field $v$ is given in terms of $q$ via the interaction kernel $G$, which is assumed to be a smooth off zero function. Here, $(x_1,x_2)^\perp=(-x_2,x_1)$. In the case that $G=\frac{1}{2\pi}\ln|\cdot|$, we arrive at the Euler equations. On the other hand, if $G=-K_0(|\lambda|| \cdot|)$, with $\lambda\neq 0$ and $K_0$ the Bessel function, the quasi--geostrophic shallow water (QGSW) equations are found. Finally, the generalized surface quasi--geostrophic (gSQG) equations appears in the case that $G=\frac{C_{\beta}}{2\pi}\frac{1}{|\cdot|^\beta}$, for $\beta\in[0,1]$ and $C_{\beta}=\frac{\Gamma\left(\frac{\beta}{2}\right)}{2^{1-\beta}\Gamma\left(\frac{2-\beta}{2}\right)}$. Note that all the previous models have a common property: the kernel $G$ is radial, which will be crucial in this work.

The Euler equations deal with uniform incompressible ideal fluids and $q$ represents the vorticity of the fluid, usually denoted by $\omega$. Yudovich solutions, that are bounded and integrable solutions, are known to exist globally in time, see \cite{MajdaBertozzi, Yudovich}. When the initial data is the characteristic function of a simply--connected bounded domain $D_0$, the solution continues being a characteristic function over $D_0$, that propagates along the flow. These are known as the  vortex patches solutions. If the initial domain is $C^{1,\alpha}$, with $0<\alpha<1$, the regularity persists for any time, see \cite{B-C,Chemin}. The only explicit simply--connected vortex patches known up to now are the Rankine vortex (the circular patch), which are stationary, and the Kirchhoff ellipses \cite{Kirchhoff}, that rotate. Later, Deem and Zabusky \cite{Deem-Zabusky} gave some numerical observations of the existence of V--states, i.e. rotating vortex patches, with $m-$fold symmetry. Using the bifurcation theory, Burbea \cite{Burbea} proved analytically the existence of these V--states close to the Rankine vortex. There has been several works concerning the V--states following the approach of Burbea: doubly--conected V--states, corotating and counter--rotating vortex pairs, non uniform rotating vortices and global bifurcation, see \cite{CastroCordobaGomezSerrano, DelaHozHmidiMateuVerdera, GHS, HMW, HmidiMateu-pairs, HmidiMateuVerdera}.

The { quasi--geostrophic shallow water equations} are found in the limit of fast rotation and weak variations of the free surface in the rotating shallow water equations, see \cite{Vallis}. In this context, the function $q$ is called the potential vorticity. The parameter $\lambda$ is known as the inverse ``Rossby deformation length'', and when it is small, it corresponds to a free surface that is nearly rigid. In the case $\lambda=0$, we recover the Euler equations. Although vortex patches solutions are better known in the Euler equations, there are also some results in the quasi--geostrophic shallow water equations. For the analogue to the Kirchhoff ellipses in these equations, we refer to the works of Dristchhel, Flierl, Polvani and Zabusky \cite{Plotka-Dristchel, Polvani, Polvani-Zabusky-Flierl}. In \cite{D-H-R}, Dristchel, Hmidi and Renault proved the existence of V--states bifurcating from the circular patch.

In the case of {  the generalized surface quasi--geostrophic,  }$q$ describes the potential temperature. This model has been proposed by C\'ordoba et al. in \cite{Cordoba} as an interpolation between the Euler equations and the surface quasi--geostrophic (SQG) equations, corresponding to $\beta=0$ and $\beta=1$, respectively. The mathematical analogy with the classical three--dimensional incompressible Euler equations can be found in \cite{Constantin-Majda-Tabak}. Some works concerning V--states in the gSQG equations are \cite{Cas0-Cor0-Gom,Cas-Cor-Gom, Hassainia-Hmidi, HmidiMateu-pairs}.

The point model for a K\'arm\'an Vortex Street consists in two infinite arrays of point vortices with opposite strength. More specifically, consider a uniformly distributed arrow of points, with same strength in every point, located in the horizontal axis, i.e., $(kl,0)$, with $l>0$ and $k\in\Z$. The other arrow contains an infinite number of points, with opposite strength, which will be parallel to the other one and with arbitrary stagger: $(a+kl,-h)$ with $a\in\R$ and $h\neq 0$. We refer to Figure 1 for a better understanding of the localization of the points. Hence, we consider the following distribution:
\begin{equation}\label{KVS}
q_0(x)=\sum_{k\in\Z}\delta_{(kl,0)}(x)-\sum_{k\in\Z}\delta_{(a+kl,-h)}(x),
\end{equation}
where $a\in\R$, $l>0$ and $h\neq 0$. If the kernel $G$ is radial, then we will show that every point translates in time, with the same constant speed. Moreover, if $a=0$ or $a=\frac{l}{2}$, the translation is parallel to the real axis. In the typical case of the Newtonian interaction, that is $G=\frac{1}{2\pi}\ln|\cdot|$, the evolution of \eqref{KVS} has been fully studied, see \cite{Aref, Lamb, Karman1, Karman2, Newton, Pierrehumbert, Rosenhead}. The problem consists in a first order Hamiltonian system and the evolution of every point is given by the following system:

\begin{align*}
\frac{d}{dt} z_m(t)=&\sum_{\substack{m\neq k\in\Z}} \frac{(z_m(t)-z_k(t))^\perp}{|z_m(t)-z_k(t)|^2}-\sum_{\substack{k\in\Z}} \frac{(z_m(t)-\tilde{z}_k(t))^\perp}{|z_m(t)-\tilde{z}_k(t)|^2},\\
\frac{d}{dt} \tilde{z}_m(t)=&\sum_{\substack{k\in\Z}}\frac{(\tilde{z}_m(t)-z_k(t))^\perp}{|\tilde{z}_m(t)-z_k(t)|^2}-\sum_{\substack{m\neq k\in\Z}}\frac{(\tilde{z}_m(t)-\tilde{z}_k(t))^\perp}{|\tilde{z}_m(t)-\tilde{z}_k(t)|^2},
\end{align*}
with initial conditions
\begin{align*}
z_m(0)=&ml,\\
\tilde{z}_m(0)=&a+ml-ih,
\end{align*}
for $m\in\Z$. It can be checked that the velocity at every point is the same, providing us with a translating motion. Indeed, if $a=0$ or $a=\frac{l}{2}$, the speed can be expressed by elementary functions, where we observe that the translation is horizontal:
\begin{align*}
V_0=&\frac{1}{2l}\coth\left(\frac{\pi h}{l}\right), \ \textnormal{ for } \  a=0,\\
V_0=&\frac{1}{2l}\tanh\left(\frac{\pi h}{l}\right),  \ \textnormal{ for } \   a=\frac{l}{2}.
\end{align*}

The works of Saffmann and Schatzman \cite{Saffman-Schatzman, Saffman-Schatzman2, Schatzman} refer to the study of these structures in the Euler equations. In fact,  they consider two infinite arrows of vortices with finite area, which have uniform vorticity inside. These are two infinite arrows of vortex patches distributed as in \eqref{KVS}. They found existence of this kind of solutions numerically and they studied their linear stability, see \cite{Saffman-Schatzman, Saffman-Schatzman2, Schatzman}.  The problem can be studied not only in the Euler equations for the full space, but in the Euler equations in the periodic setting. A theory for the Euler equations in an infinite strip can be found in \cite{Beichman-Denisov}. In \cite{VladimirGryanikBorthOlbers}, Gryanik, Borth and Olbers studied the quasi--geostrophic K\'arm\'an Vortex Street in two--layer fluids.

\begin{figure}[h]
\begin{center}
\includegraphics[width=0.6\textwidth]{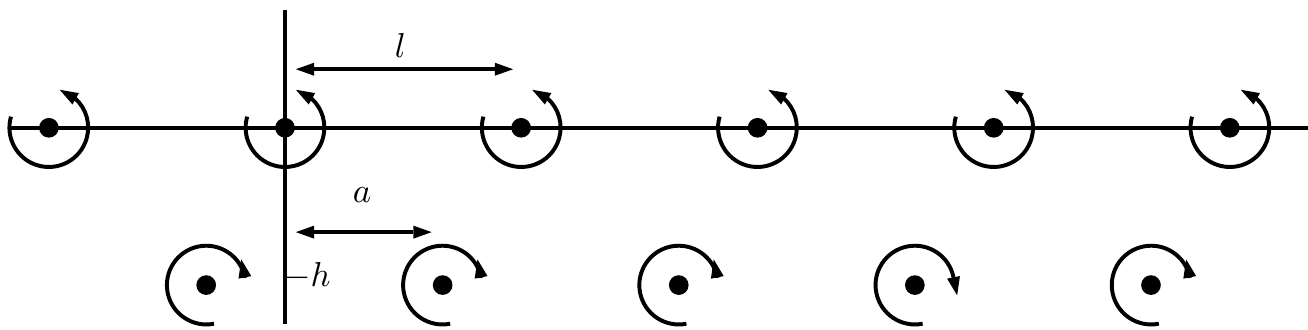}
\caption{K\'arm\'an Vortex Street located at the points $(kl,0)$ and $(a+kl,-h)$,  with  $l>0$, $a\in\R$, $h\neq 0$, and $k\in\Z$.}
\label{fig}
\end{center}
\end{figure}

The aim of this work is to find analitically solutions to the model proposed by Saffman and Schatzman in \cite{Saffman-Schatzman, Saffman-Schatzman2, Schatzman}. We will do it not only for the Euler equations, but also for other inviscid incompressible models.  Here, we follow the idea of Hmidi and Mateu in \cite{HmidiMateu-pairs}, where they desingularize a vortex pairs obtaining a pair of vortex patches that rotate or translate (depending on the strength of the points). The plan is the following. For $\varepsilon>0$ and $l>0$, consider
\begin{equation}\label{KVS2}
q_{0,\varepsilon}(x)=\frac{1}{\pi\varepsilon^2}\sum_{k\in\Z}{\bf{1}}_{\varepsilon D_1+kl}(x)-\frac{1}{\pi\varepsilon^2}\sum_{k\in\Z}{\bf{1}}_{\varepsilon D_2+kl}(x_1,x_2),
\end{equation}
where $D_1$ and $D_2$ are simply--connected bounded domains. In the case $|D_1|=|\D|$ and $D_2=-D_1+a-ih$, with $a=0$ or $a=\frac{l}{2}$, and $h\neq 0$, we find the K\'arm\'an Vortex Street \eqref{KVS} in the limit $\varepsilon\rightarrow 0$. This means
$$
\lim_{\varepsilon\rightarrow 0}\, \left\{\frac{1}{\pi\varepsilon^2}\sum_{k\in\Z}{\bf{1}}_{\varepsilon D_1+kl}(x)-\frac{1}{\pi\varepsilon^2}\sum_{k\in\Z}{\bf{1}}_{-\varepsilon D_1+a+kl-ih}(x_1,x_2)\right\}=\sum_{k\in\Z}\delta_{(kl,0)}(x)-\sum_{k\in\Z}\delta_{(a+kl,-h)}(x),
$$
in the distribution sense. Then, we connect the vortex patch model \eqref{KVS2} with the point vortex model \eqref{KVS}. In the following, we will refer to K\'arm\'an Vortex Street or K\'arm\'an Point Vortex Street when having the point vortex model \eqref{KVS}. Otherwise, we denote by K\'arm\'an Vortex Patch Street in the case of \eqref{KVS2}. Some relation between the two domains is needed, and the suitable one is mentioned before: $D_2=-D_1+a-ih$. Assuming that $q(t,x)=q_0(x-Vt)$, for some $V\in\R$, and using a conformal map $\phi:\T\rightarrow \partial D_1$, we arrive at the following equation

\begin{equation*}
F(\varepsilon,f,V)(w):=\textnormal{Re}\left[\left\{\overline{I(\varepsilon,f)(w)}- \overline{V}\right\}{w}{\phi'(w)}\right]=0, \quad w\in\T,
\end{equation*}
where
\begin{align*}
I(\varepsilon,f)(w):=&-\frac{1}{\pi\varepsilon}\sum_{k\in\Z}\int_{\T} G(|\varepsilon(\phi(w)-\phi(\xi))-kl|)\phi'(\xi)\, d\xi\\
&-\frac{1}{\pi\varepsilon}\sum_{k\in\Z}\int_{\T} G(|\varepsilon(\phi(w)+\phi(\xi))-a-kl+ih|)\phi'(\xi)\, d\xi.
\end{align*}
Let us explain the meaning of $f$. Assume that the conformal map is a perturbation of the identity in the following way 
\begin{equation}\label{confmap1}
\phi(w)=i \left(w+\varepsilon f(w)\right), \quad f(w)=\sum_{n\geq 1}a_nw^{-n}, \quad a_n\in\R, w\in\T.
\end{equation}
for $G=\frac{1}{2\pi}\ln|\cdot|$ and $G=-K_0(|\lambda||\cdot|)$. Whereas, it will be consider as
\begin{equation}\label{confmap2}
\phi(w)=i \left(w+\frac{\varepsilon}{G(\varepsilon)} f(w)\right), \quad f(w)=\sum_{n\geq 1}a_nw^{-n}, \quad a_n\in\R, w\in\T.
\end{equation}
for more singular kernels, such as $G=\frac{C_{\beta}}{2\pi}\frac{1}{|\cdot|^\beta}$, for $\beta\in(0,1)$. Moreover, we have that $F(0,0,V_0)(w)=0$, for any $w\in\T$. Here, $V_0$ is the speed of the point model \eqref{KVS}.  The nonlinear functional $F$ is well-defined from $\R\times X_{\alpha}\times \R$ to $\tilde{Y}_\alpha$, where
\begin{align*}
X_{\alpha}&=\left\{f\in C^{1,\alpha}(\T), \quad f(w)=\sum_{n\geq 1}a_nw^{-n},\,  a_n\in\R\right\},\\
\tilde{Y}_{\alpha}&=\left\{f\in C^{0,\alpha}(\T), \quad f(e^{i\theta})=\sum_{n\geq 1}a_n\sin(n\theta),\,  a_n\in\R\right\},
\end{align*}
for some $\alpha\in(0,1)$. However, $\partial_f {F}(0, 0, V)$ is not an isomorphism in these spaces. Defining $V$ as a function of $\varepsilon$ and $f$ in the following way
\begin{align}\label{V}
V(\varepsilon,f)=&\frac{\int_{\T}\overline{ I(\varepsilon,f)(w)}{w}{\phi'(w)}(1-\overline{w}^2)dw}{\int_{\T} {w}{\phi'(w)}(1-\overline{w}^2)dw},
\end{align}
then, $F$ is also well-defined from $\R\times X_{\alpha}$ to $Y_{\alpha}$, where
$$
Y_{\alpha}=\left\{f\in C^{0,\alpha}(\T), \quad f(e^{i\theta})=\sum_{n\geq 2}a_n\sin(n\theta), \, a_n\in\R\right\}.
$$
Indeed, $\partial_f {F}(0, 0, V)$ is an isomorphism in these spaces. Then, we fix the velocity $V$ depending on $\varepsilon$ and $f$ as in \eqref{V}. In this way, the Implicit Function theorem can be applied in order to desingularize the point model \eqref{KVS} obtaining our main result:
\begin{theo}
Consider $G=\frac{1}{2\pi}\ln|\cdot|$, $G=-K_0(|\lambda||\cdot|)$, or $G=\frac{C_{\beta}}{2\pi}\frac{1}{|\cdot|^\beta}$, for $\lambda\neq 0$ and $\beta\in(0,1)$. Let $h, l\in\R$, with $h\neq 0$ and $l>0$, and $a=0$ or $a=\frac{l}{2}$. Then, there exists a $C^1$ simply--connected bounded domain $D^{\varepsilon}$ such that
\begin{equation*}
q_{0,\varepsilon}(x)=\frac{1}{\pi\varepsilon^2}\sum_{k\in\Z}{\bf 1}_{\varepsilon D^{\varepsilon}+kl}(x)-\frac{1}{\pi\varepsilon^2}\sum_{k\in\Z}{\bf 1}_{-\varepsilon D^{\varepsilon}+a-ih+kl}(x),
\end{equation*}
defines a horizontal translating solution of \eqref{Generalsystem}, with constant  speed, for any $\varepsilon\in(0,\varepsilon_0)$ and small enough $\varepsilon_0>0$.
\end{theo}
The proof of the theorem has to be adapted to each case. For the cases $G=\frac{1}{2\pi}\ln|\cdot|$ and $G=-K_0(|\lambda||\cdot|)$, the kernel has a logarithm singularity at the origin, which will play an important role in the choice of the scaling of the conformal map \eqref{confmap1}. Otherwise, if $G=\frac{C_{\beta}}{2\pi}\frac{1}{|\cdot|^\beta}$, for $\beta\in(0,1)$, or for more general kernels, the scaling of the conformal map will depend on the singularity at the origin as depicted in \eqref{confmap2}.

This work is organized as follows. In Section 2, we introduce some preliminary results about the $N$--vortex problem and, in particular, the point model for the K\'arm\'an Vortex Street \eqref{KVS} with general interactions.  Section \ref{Sec2} refers to the Euler and QGSW equations, where the singularity of the kernel is logarithmic.  We will deal with the general case in Section \ref{Sec3}, following the same ideas as in the Euler equations. The gSQG equations will become a particular case of this study. Finally, Appendix \ref{Ap-specialfunctions} and Appendix \ref{Ap-potentialtheory} will be devoted to provide some definitions and properties concerning special functions and complex integrals.

Let us end this part by summarizing some notation to be used along the paper. We will denote the unit disc by $\D$ and its boundary by $\T$. The integral
$$
\int_{\T} f(\xi)\, \, d\xi,
$$
denotes the usual complex integral, for some complex function $f$. Moreover, we will use the symmetry sums defined by
\begin{equation}\label{sym-sum}
\sum_{k\in\Z}a_k=\lim_{K\rightarrow +\infty}\sum_{|k|\leq K}a_k,
\end{equation}
except being specified.

\section{The N--vortex problem}\label{Sec-NVP}

The $N$--vortex problem consists in the study of the evolution of a set of points that interact according to some laws. Originally, the Newtonian interaction for the $N$--vortex problem is considered. But here, we start by assuming that the interaction between the points is due to a general function $G:\R\rightarrow\R$, which is smooth off zero. The problem is a first order Hamiltonian system of the form
\begin{equation}\label{dyn-syst}
\frac{d}{dt} z_m(t)=\sum_{\substack{k=1 \\  k\neq m}}^{N}\Gamma_k \nabla^\perp_{z_m} G(|z_m(t)-z_k(t)|),
\end{equation}
with some initial conditions at $t=0$ and where $m=1,\dots, N$. Moreover, $z_k\neq z_m$ if $k\neq m$, are $N$--points located in the plane $\R^2$. The constants $\Gamma_k$ mean the strength of each point and the interaction between them is due to $G$. In the case that $G=\frac{1}{2\pi}\ln|\cdot|$, we arrive at the typical $N$--vortex problem:
$$
\frac{d}{dt} z_m(t)=\sum_{\substack{k=1 \\  k\neq m}}^{N}\Gamma_k \frac{(z_m(t)-z_k(t))^\perp}{|z_m(t)-z_k(t)|^2}.
$$
First, we deal with the evolution of two points. We can check that they rotate or translate depending on their strengh: $\Gamma_1$ and $\Gamma_2$. Later, we will move on the configuration that we concern: an infinite number of points with a periodic pattern in space. For more details about the $N$--vortex problem in the case of the Newtonian interaction see \cite{Newton}.

\subsection{Finite configurations}

There are some stable situations yielding to steady configurations. Here, we  illustrate the evolution of $2$--points since the same idea is used later in the case of periodic configurations.

\begin{pro}
Let us consider two initial points $z_1(0)$ and $z_2(0)$, with $z_1(0)\neq z_2(0)$, located in the real axis, with strengths $\Gamma_1$ and $\Gamma_2$ respectively. We have:
\begin{enumerate}
\item If $\Gamma_1+\Gamma_2\neq 0$ and $\Gamma_1z_1(0)+\Gamma_2z_2(0)=0$, then the solution behaves as 
$z_k(t)=e^{i\Omega t}z_k(0),$ for $k=1,2$, with $\Omega=\frac{(\Gamma_1+\Gamma_2)G'(|z_1(0)-z_2(0)|)}{|z_1(0)-z_2(0)|}$.
\item If $\Gamma_1+\Gamma_2=0$, then $z_k(t)=z_k(0)+U t$, for $k=1,2$, with $U=i\Gamma_2 G'(|z_1(0)-z_2(0)|)\textnormal{sign}(z_1(0)-z_2(0))$.
\end{enumerate}
\end{pro}
\begin{proof}
According to \eqref{dyn-syst}, the evolution of the two points is given by the following system:
\begin{equation}
\left\{
\begin{array}{l}\label{2-vortex}
\frac{d}{dt}z_1(t)=i\Gamma_2G'(|z_1(t)-z_2(t)|)\frac{z_1(t)-z_2(t)}{|z_1(t)-z_2(t)|},\\
\frac{d}{dt}z_2(t)=-i\Gamma_1G'(|z_1(t)-z_2(t)|)\frac{z_1(t)-z_2(t)}{|z_1(t)-z_2(t)|}.
\end{array}
\right.
\end{equation}

\noindent
\medskip
{\bf (1)} 
By the above system, it is clear that $\frac{d}{dt}\left(\Gamma_1 z_1(t)+\Gamma_2 z_2(t)\right)=0$, which implies that 
$$
\Gamma_1 z_1(t)+\Gamma_2 z_2(t)=\Gamma_1 z_1(0)+\Gamma_2 z_2(0)=0.
$$
Assuming that $z_1(t)=e^{i\Omega t}z_1(0)$, and using the above equation, one arrives at $z_2(t)=e^{i\Omega t}z_2(0)$. Then, the system  \eqref{2-vortex} yields
$$
\left\{
\begin{array}{l}
i\Omega e^{i\Omega t}z_1(0)=ie^{i\Omega t}\Gamma_2G'(|z_1(0)-z_2(0)|)\frac{z_1(0)-z_2(0)}{|z_1(0)-z_2(0)|},\\
i\Omega e^{i\Omega t}z_2(0)=-ie^{i\Omega t}\Gamma_1G'(|z_1(0)-z_2(0)|)\frac{z_1(0)-z_2(0)}{|z_1(0)-z_2(0)|}.
\end{array}
\right.
$$
Since $z_1(0)$ and $z_2(0)$ are located in the real axis, one has that $z_1(0)-z_2(0)\in\R$, and then subtracting the above two equations amounts to
$$
\Omega=\frac{(\Gamma_1+\Gamma_2)G'(|z_1(0)-z_2(0)|)}{|z_1(0)-z_2(0)|}.
$$

\noindent
\medskip
{\bf (2)}
In this case, we have that $\frac{d}{dt}\left(z_1(t)- z_2(t)\right)=0$, and thus
$$
z_1(t)- z_2(t)=z_1(0)- z_2(0).
$$
As a consequence, \eqref{2-vortex} agrees with
$$
\left\{
\begin{array}{l}
\frac{d}{dt}z_1(t)=i\Gamma_2G'(|z_1(0)-z_2(0)|)\textnormal{sign}(z_1(0)-z_2(0)),\\
\frac{d}{dt}z_2(t)=i\Gamma_2G'(|z_1(0)-z_2(0)|)\textnormal{sign}(z_1(0)-z_2(0)),\\
\end{array}
\right.
$$
which can be solved as
$$
\left\{
\begin{array}{l}
z_1(t)=z_1(0)+i\Gamma_2G'(|z_1(0)-z_2(0)|)\textnormal{sign}(z_1(0)-z_2(0))t,\\
z_2(t)=z_2(0)+i\Gamma_2G'(|z_1(0)-z_2(0)|)\textnormal{sign}(z_1(0)-z_2(0))t.
\end{array}
\right.
$$
\end{proof}
The above result gives us that two vortex points with $\Gamma_1+\Gamma_2\neq 0$, have a rotating evolution. Otherwise, they {\it translate}. From now on, we refer that a structure {\it translates} when the evolution of every point (or every patch, in the case of \eqref{Generalsystem}) is a translation, with the same constant speed.

In the usual $N$--vortex problem, meaning $G=\frac{1}{2\pi}\ln|\cdot|$, the above result is well--known. Here, we have seen that if the interaction of the points is due to a kernel that is radial, we get the same evolution. 

\subsection{Periodic setting}

This section deals with the evolution of two infinite arrays of points with opposite strength, which are periodic in space. More specifically, the points of the first arrow, which have the same strength, will be located in the horizontal axis. Let us assume that we have one point at the origin, and the next one differs of it a distance $l$. This means that we have the points $(kl,0)$, for $l>0$ and $k\in\Z$. The second arrow, with opposite strength to the previous arrow, will be parallel to the horizontal axis but with a height $h\neq 0$, having the following distribution of points: $(a+kl, -h)$, for $a\in\R$ and $k\in\Z$. For the moment, let us consider that $a$ is any real number. 

Then, we focus on 
\begin{equation}\label{PointVortex}
q(x)=\sum_{k\in\Z}\delta_{(kl,0)}(x)-\sum_{k\in\Z}\delta_{(a+kl,-h)}(x),
\end{equation}
with $h\neq 0$, $l>0$ and $a\in\R$. In the following results, we check that the above initial configuration translate when $G=\frac{1}{2\pi}\ln|\cdot|$, $G=-K_0(|\lambda||\cdot|)$, $G=\frac{C_{\beta}}{2\pi}\frac{1}{|\cdot|^\beta}$ for $\beta\in(0,1)$, or for $G$ satisfying some general conditions. Moreover, if $a=0$ or $a=\frac{l}{2}$, the translation is horizontal.

We are going to differentiate two cases depending on the behavior of the interaction $G$ at infinity. This is important in order to give a meaning to the infinite sum coming from \eqref{PointVortex}, whose equations are given by

\begin{align*}
\frac{d}{dt} z_m(t)=&\sum_{\substack{m\neq k\in\Z}} G'(|z_m(t)-z_k(t)|)\frac{(z_m(t)-z_k(t))^\perp}{|z_m(t)-z_k(t)|}-\sum_{k\in\Z} G'(|z_m(t)-\tilde{z}_k(t)|)\frac{(z_m(t)-\tilde{z}_k(t))^\perp}{|z_m(t)-\tilde{z}_k(t)|},\\
\frac{d}{dt} \tilde{z}_m(t)=&\sum_{k\in\Z} G'(|\tilde{z}_m(t)-z_k(t)|)\frac{(\tilde{z}_m(t)-z_k(t))^\perp}{|\tilde{z}_m(t)-z_k(t)|}-\sum_{m\neq k\in\Z} G'(|\tilde{z}_m(t)-\tilde{z}_k(t)|)\frac{(\tilde{z}_m(t)-\tilde{z}_k(t))^\perp}{|\tilde{z}_m(t)-\tilde{z}_k(t)|},
\end{align*}
with initial conditions
\begin{align*}
z_m(0)=&ml,\\
\tilde{z}_m(0)=&a+ml-ih,
\end{align*}
for $m\in\Z$. Then, we refer to the critical case in the case of the Newtonian interaction
$$
G=\frac{1}{2\pi}\ln|\cdot|.
$$
Here, we must use the structure of the logarithm in order to have a convergence sum. Note that here we need to use strongly the symmetry sum. Otherwise, the subcritical cases will use the faster decay of $G$ at infinity as it is the case of the QGSW or gSQG interactions.

\noindent
\medskip
$\bullet$ {\it Critical case:} Let us first show the result for the Newtonian interaction, that is, $G=\frac{1}{2\pi}\ln|\cdot|$. Here, we denote $\omega$ to $q$ to emphasize that we are working with the vorticity.

\begin{pro}\label{Prop-PV}
Given the point vortex street \eqref{PointVortex} with $G=\frac{1}{2\pi}\ln|\cdot|$, for $h\neq 0$, $l>0$ and $a\in\R$, then the street is moving with the following constant velocity speed 
\begin{equation}\label{velocity_PV}
V_0=\frac{1}{2l i}\overline{\cot\left(\frac{\pi(ih-a)}{l}\right)}.
\end{equation}
In the case that $a=0$ or $a=\frac{l}{2}$, the translation is parallel to the horizontal axis with velocity
\begin{align*}
V_0=&\frac{1}{2l}\coth\left(\frac{\pi h}{l}\right), \ \textnormal{ for } \  a=0,\\
V_0=&\frac{1}{2l}\tanh\left(\frac{\pi h}{l}\right),  \ \textnormal{ for } \   a=\frac{l}{2}.
\end{align*}
\end{pro}
\begin{rem}
If $\omega_{\kappa}(x,y)=\kappa \omega(x,y),$ with $\kappa\in\R$ and $\omega$ given by \eqref{PointVortex}, then the velocity of the street is $V_{0,\kappa}=\kappa V_0$.
\end{rem}
\begin{proof}
Define
$$
\omega_K(x)=\sum_{|k|\leq K}\delta_{(kl,0)}(x)-\sum_{|k|\leq K}\delta_{(a+kl,-h)}(x).
$$
The idea is to consider $K\rightarrow +\infty$, getting $\omega_K\rightarrow \omega$ in the distribution sense. The associated stream function to $\omega_K$ is given by
\begin{equation}\label{streamfunction_PV}
\psi_K(x)=\frac{1}{2\pi}\sum_{|k|\leq K}\ln|x-kl|-\frac{1}{2\pi}\sum_{|k|\leq K}\ln|x-a-kl+ih|,
\end{equation}
where we are using complex notation. In order to pass to the limit, we need to use the structure of the logarithm. Let us work with the sum in the following way
\begin{align*}
\sum_{|k|\leq K}\ln|x-a-kl+ih|=&\ln\left|\prod_{|k|\leq K}(x-a-kl+ih)\right|\\
=&\ln\left|(x-a+ih)\prod_{k=1}^{K}\left((x-a+ih)^2-k^2l^2\right)\right|\\
=&\ln\left|\frac{\pi(x-a+ih)}{l}\prod_{k=1}^{K}\left(1-\frac{(x-a+ih)^2}{k^2l^2}\right)\right|+\ln\left|\frac{l}{\pi}\prod_{k=1}^{K}k^2l^2\right|.
\end{align*}
Then, we have
\begin{align*}
\lim_{K\rightarrow \infty}\psi_K(x)=\lim_{K\rightarrow \infty}&\left\{\frac{1}{2\pi}\ln\left|\frac{\pi x}{l}\prod_{k=1}^{K}\left(1-\frac{x^2}{k^2l^2}\right)\right|\right.\\
&\left.-\frac{1}{2\pi}\ln\left|\frac{\pi(x-a+ih)}{l}\prod_{k=1}^{K}\left(1-\frac{(x-a+ih)^2}{k^2l^2}\right)\right|\right\}.
\end{align*}
Using the product expression for the sine, that is
\begin{align}\label{sine}
\sin (\pi x)=\pi x\prod_{k\geq 1}\left(1-\frac{x^2}{k^2}\right),
\end{align}
we get that
\begin{equation*}
\psi(x):=\lim_{K\rightarrow \infty}\psi_K(x)=\frac{1}{2\pi}\ln\left|\sin\left(\frac{\pi x}{l}\right)\right|-\frac{1}{2\pi}\ln\left|\sin\left(\frac{\pi(x-a+ih)}{l}\right)\right|.
\end{equation*}
In this way, we achieve that the sum in \eqref{streamfunction_PV} converges. In the same way, the corresponding velocity agrees with
$$
v_K(x)=\frac{i}{2\pi}\sum_{|k|\leq K}\frac{x-kl}{|x-kl|^2}-\frac{i}{2\pi}\sum_{|k|\leq K}\frac{x-a-kl+ih}{|x-a-kl+ih|^2},
$$
where $x$ is none of the points vortex. In each of the points, the velocity is given by
\begin{align*}
v_K(ml)=&\frac{i}{2\pi}\sum_{k\neq m,|k|\leq K}\frac{ml-kl}{|ml-kl|^2}-\frac{i}{2\pi}\sum_{|k|\leq K}\frac{ml-a-kl+ih}{ml-a-kl+ih|^2},\\
v_K(a-ih+ml)=&\frac{i}{2\pi}\sum_{|k|\leq K}\frac{a-ih+ml-kl}{|x-kl|^2}-\frac{i}{2\pi}\sum_{k\neq m,|k|\leq K}\frac{ml-kl}{|x-a-kl+ih|^2},
\end{align*}
for any $m\in\Z$. We define $v$ as the limit of the above function.

First, let us show that the velocity at every point is the same. We begin with the first arrow
\begin{align*}
v(ml)=&\frac{i}{2\pi}\sum_{m\neq k\in \Z}\frac{ml-kl}{|ml-kl|^2}-\frac{i}{2\pi}\sum_{k\in\Z}\frac{ml-a-kl+ih}{|ml-a-kl+ih|^2}\\
=&-\frac{i}{2\pi}\sum_{0\neq k\in \Z}\frac{kl}{|kl|^2}+\frac{i}{2\pi}\sum_{k\in\Z}\frac{a+kl-ih}{|a+kl-ih|^2}\\
=&\frac{i}{2\pi}\sum_{k\in\Z}\frac{a+kl-ih}{|a+kl-ih|^2}\\
=&v(0).
\end{align*}
Note that
$$
\sum_{0\neq k\in \Z}\frac{kl}{|kl|^2}=0,
$$
since we are using the symmetry sum \eqref{sym-sum}. For the second arrow, we obtain
\begin{align*}
v(a-ih+ml)=&\frac{i}{2\pi}\sum_{k\in\Z}\frac{a-ih+ml-kl}{|a-ih+ml-kl|^2}-\frac{i}{2\pi}\sum_{m\neq k\in\Z}\frac{ml-kl}{|ml-kl|^2}\\
=&\frac{i}{2\pi}\sum_{k\in\Z}\frac{a+kl-ih}{|a+kl-ih|^2}-\frac{i}{2\pi}\sum_{0\neq k\in \Z}\frac{kl}{|kl|^2}\\
=& v(0).
\end{align*}
Then, the velocity speed of the street is given by
$$
V_0=\frac{i}{2\pi}\sum_{k\in\Z}\frac{a+kl-ih}{|a+kl-ih|^2}-\frac{i}{2\pi}\sum_{0\neq k\in \Z}\frac{kl}{|kl|^2}=\frac{i}{2\pi}\sum_{k\in\Z}\frac{a+kl-ih}{|a+kl-ih|^2}.
$$
In order to find a better expression for $V$, let us come back to the stream function. Using 
\begin{align*}
\nabla \ln |\sin(x)|=\frac{\sin x_1\cos x_1+i\sinh x_2 \cosh x_2}{|\sin x|^2}=\overline{\cot x},
\end{align*}
we achieve
$$
V_0=\frac{1}{2l i}\overline{\cot\left(\frac{\pi(ih-a)}{l}\right)}.
$$
Let us now work with $a=\frac{l}{2}$. Using the definition of the complex cotangent, we obtain
\begin{align*}
\cot\left(\frac{\pi(ih-\frac{l}{2})}{l}\right)=-i\frac{\sinh\left(\frac{\pi h}{l}\right)}{\cosh\left(\frac{\pi h}{l}\right)}=-i\tanh\left(\frac{\pi h}{l}\right),
\end{align*}
which is the announced expression for the velocity. The same idea can be applied to get the expression when $a=0$.
\end{proof}
For the cases $a=0$ and $a=\frac{l}{2},$  we notice that the velocity increases as $h$ goes to 0. Moreover, considering $a=\frac{l}{2}$ and $h\rightarrow 0$ in the above proposition, one obtains the following corollary.
\begin{cor}
The vortex arrow given by
$$
\omega(x)=\sum_{k\in\Z}\delta_{(kl,0)}(x)-\sum_{k\in\Z}\delta_{(\frac{l}{2}+kl,0)}(x),
$$
is stationary, for any $l>0$.
\end{cor}

Similar ideas can be applied to find that a horizontal arrow of points with the same strength is stationary.

\begin{pro}\label{Prop-euler_statarrow}
The vortex arrow given by
$$
\omega(x)=\sum_{k\in\Z}\delta_{(a+kl,-h)}(x),
$$
is stationary, for any $a\in\R$ and $h\in\R$.
\end{pro}

\noindent
\medskip
$\bullet$ {\it Subcritical case:} We finish this section by showing the result for faster decays interactions. This case will cover the QGSW and gSQG interactions: $G=-K_0(|\lambda||\cdot|)$ or $G=\frac{C_{\beta}}{2\pi}\frac{1}{|\cdot|^\beta}$ for $\beta\in(0,1)$. The result reads as follows.
\begin{pro}\label{Gen-point}
Let $G:\R^2\rightarrow \R$ be a smooth off zero function satisfying
\begin{enumerate}
\item[(H1)] $G$ is radial such that $G(x)=\tilde{G}(|x|)$,
\item[(H2)] there exists $R>0$ and $\beta_1\in(0,1]$ such that $|\tilde{G}'(r)|\leq \frac{C}{r^{1+\beta_1}}$, for $r\geq R$.
\end{enumerate}
Then,
\begin{equation}\label{PointVortex-gen}
q(x)=\sum_{k\in\Z}\delta_{(kl,0)}(x)-\sum_{k\in\Z}\delta_{(a+kl,-h)}(x),
\end{equation}
with $h\neq 0$, $l>0$ and $a\in\R$, translates with constant velocity speed
\begin{equation}\label{V_0-gen}
V_0=i\sum_{k\in\Z} G'(|a+kl-ih|)\frac{a+kl-ih}{|a+kl-ih|}.
\end{equation}
In the case $a=0$ or $a=\frac{l}{1}$, the translation is parallel to the horizontal axis.
\end{pro}
\begin{rem}
From now on, we will assume that $G$ is radial via (H1), we will write $G$ for $\tilde{G}$ when there is no confusion in order to simplify notation.
\end{rem}
\begin{rem}
The second hypothesis is required to give a meaning to the infinite sum, which converges absolutely. This condition could be weakened by assuming
$$
\sum_{k\in\Z}\left|G'(a+kl-ih)\right|<+\infty.
$$
\end{rem}
\begin{proof}
As in the previous models, the velocity at the points is given by
\begin{align*}
-iv(ml)=&\sum_{m\neq k\in\Z} G'(|ml-kl|)\frac{ml-kl}{|ml-kl|}\\
&-\sum_{k\in\Z} G'(|ml-a-kl+ih|)\frac{ml-a-kl+ih}{|ml-a-kl+ih|},\\
-iv(a+ml-ih)=&\sum_{k\in\Z} G'(|a+ml-ih-kl|)\frac{a+ml-ih-kl}{|a+ml-ih-kl|}\\
&-\sum_{m\neq k\in\Z} G'(|ml-kl|)\frac{ml-kl}{|ml-kl|},
\end{align*}
for $m\in\Z$. The above sums are converging due to the second assumption. We can check that the velocity is the same at every point of the street:
\begin{align*}
-iv(ml)=&\sum_{m\neq k\in\Z} G'(|ml-kl|)\frac{ml-kl}{|ml-kl|}\\
&-\sum_{k\in\Z} G'(|ml-a-kl+ih|)\frac{ml-a-kl+ih}{|ml-a-kl+ih|}\\
=&\sum_{0\neq k\in\Z} G'(|kl|)\frac{kl}{|kl|}+\sum_{k\in\Z} G'(|a+kl-ih|)\frac{a+kl-ih}{|a+kl-ih|}\\
=&\sum_{k\in\Z} G'(|a+kl-ih|)\frac{a+kl-ih}{|a+kl-ih|}\\
=&-iv(0),\\
-iv(a+ml-ih)=&\sum_{k\in\Z} G'(|a+ml-ih-kl|)\frac{a+ml-ih-kl}{|a+ml-ih-kl|}\\
&-\sum_{m\neq k\in\Z} G'(|ml-kl|)\frac{ml-kl}{|ml-kl|}\\
=&\sum_{k\in\Z} G'(|a-ih-kl|)\frac{a-ih-kl}{|a-ih-kl|}-\sum_{0\neq k\in\Z} G'(|kl|)\frac{kl}{|kl|}\\
=&-iv(0).
\end{align*}
Then,
$$
V_0=v(0)=i\sum_{k\in\Z} G'(|a+kl-ih|)\frac{a+kl-ih}{|a+kl-ih|}.
$$
If $a=0$ or $a=\frac{l}{2}$, one has that
$$
\sum_{k\in\Z} G'(|a+kl-ih|)\frac{a+kl}{|a+kl-ih|}=0,
$$
and the translation is in the horizontal direction.
\end{proof}
In the general case, we also have that an array is stationary.
\begin{pro}
Let $G:\R^2\rightarrow \R$ be a smooth off zero function satisfying (H1)-(H2), then
\begin{equation*}
q(x)=\sum_{k\in\Z}\delta_{(a+kl,-h)}(x),
\end{equation*}
is stationary for any $a\in\R$ and $h\in\R$.
\end{pro}
It is clear that $G=\frac{C_{\beta}}{2\pi}\frac{1}{|\cdot|^\beta}$, for $\beta\in(0,1)$, satisfies the hypothesis of the above results. In the case of the QGSW interaction, we obtain similar results. In this case, the stream function associated to \eqref{PointVortex} is given by
\begin{equation}\label{QGSW-psi1}
\psi(x)=-\frac{1}{2\pi}\sum_{k\in \Z} K_0(\lambda |x-kl|)+\frac{1}{2\pi}\sum_{k\in\Z} K_0(\lambda |x-a-kl+ih|).
\end{equation}
The definition and some properties of the Bessel functions can be found in Appendix \ref{Ap-specialfunctions}. The above sum is convergent due to the behavior of $K_0$ at infinity, which is exponential:
\begin{equation*}
K_0(z)\sim \sqrt{\frac{\pi}{2z}}e^{-z},\quad |\textnormal{arg}(z)|<\frac{3}{2}\pi.
\end{equation*}
There is another representation of the stream function given in \cite{VladimirGryanikBorthOlbers}, where the periodicity structure is emphasized:
\begin{align}\label{QGSW-psi2}
\psi(x_1,x_2)=&-\frac{1}{a}\sum_{k\in \Z}\frac{\exp\left(-\sqrt{\left(\frac{2\pi k}{a}\right)^2+\lambda^2}|x_2|\right)}{\sqrt{\left(\frac{2\pi k}{a}\right)^2+\lambda^2}}\cos\left(\frac{2\pi k}{a}x_1\right)\nonumber\\
&+\frac{1}{a}\sum_{k\in \Z}\frac{\exp\left(-\sqrt{\left(\frac{2\pi k}{a}\right)^2+\lambda^2}|x_2+h|\right)}{\sqrt{\left(\frac{2\pi k}{a}\right)^2+\lambda^2}}\cos\left(\frac{2\pi k}{a}(x_1-a)\right).
\end{align}
Then, we state the result concerning the QGSW interaction.
\begin{pro}\label{Prop-PV-QGSW}
Given the point vortex street \eqref{PointVortex}, with $h\neq 0$, $l>0$ and $a\in\R$, then the street translates with the following constant velocity speed 
\begin{equation}\label{velocity_PV-QGSW}
V_0=\frac{\lambda i}{2\pi}\sum_{ k\in\Z}K_1(\lambda|a+kl-ih|)\frac{a+kl-ih}{|a+kl-ih|}.
\end{equation}
In the case that $a=0$ or $a=\frac{l}{2}$, the translation is parallel to the horizontal axis.
\end{pro}

\section{Periodic patterns in the Euler and QGSW equations}\label{Sec2}
This section is devoted to show the full construction of the K\'arm\'an Vortex Street structures in the Euler equations. 
Instead of considering two arrows of points as in Section \ref{Sec-NVP}, we consider two infinite arrows of patches distributed in the same way than the arrows of points \eqref{PointVortex}. We will refer to this configuration in the Euler equations as K\'arm\'an Vortex Patch Street.

In the case of arrows of points, we showed in the last section that they translate. Here, we want to find a similar evolution in the Euler equations. Since these structures are periodic is space, first we will have to look for the green function associated to the $-\Delta$ operator in $\T\times\R$, which will come as an infinite sum of functions. This infinite sum can be expressed in terms of elementary functions, which helps us in the computations. Once we have the equation that will characterize the K\'arm\'an Vortex Patch Street, we will have to deal with the Implicit Function theorem. Hence, a desingularization of the K\'arm\'an Point Vortex Street will show the existence of these structures in terms of finite area domains that translate.

At the end of this section, we will analyze the case of the QGSW equations, which will follow similarly. Let us focus now in the Euler equations:

\begin{eqnarray*}	           
       \left\{\begin{array}{ll}
          	\omega_t+(v\cdot \nabla) \omega=0, &\text{ in $[0,+\infty)\times\mathbb{R}^2$}, \\
         	 v=K*\omega,&\text{ in $[0,+\infty)\times\mathbb{R}^2$}, \\
         	 \omega(t=0,x)=\omega_0(x),& \text{ with $x\in\mathbb{R}^2$}.
       \end{array}\right.
\end{eqnarray*}
The second equation links the velocity to the vorticity through the Biot--Savart  law, where $K(x)=\frac{1}{2\pi}\frac{x^\perp}{|x|^2}$ and $x^\perp=(-x_2,x_1)$. We denote by $\psi$ the stream function, which verifies $v=\nabla^\perp \psi$.

From now on we will use complex notation in order to simplify the computations. Then, we identify $(x_1,x_2)\in\R^2$ with $x_1+ix_2\in\C$. In the same way, $x^\perp=ix$. Moreover, the gradient operator in $\R^2$ can be identified with the Wirtinger derivative, i.e.,
\begin{equation}\label{gradient-complex}
\nabla=2\partial_{\overline{z}}, \quad \partial_{\overline{z}}\varphi(z):=\frac12\left(\partial_{1}\varphi(z)+i\partial_{2} \varphi(z)\right),
\end{equation}
for a complex function $\varphi$.

\subsection{Velocity of the K\'arm\'an Vortex Patch Street}
Consider the initial condition given by
\begin{align}\label{FAPointVortex}
\omega_0(x_1,x_2)=&\frac{1}{\pi}\sum_{k\in\Z}{\bf{1}}_{D_1}(x_1-kl,x_2)-\frac{1}{\pi}\sum_{k\in\Z}{\bf{1}}_{D_2}(x_1-kl,x_2)\nonumber\\
=&\frac{1}{\pi}\sum_{k\in\Z}{\bf{1}}_{D_1+kl}(x_1,x_2)-\frac{1}{\pi}\sum_{k\in\Z}{\bf{1}}_{D_2+kl}(x_1,x_2),
\end{align}
where $D_1$ and $D_2$ are simply--connected bounded domains such that $|D_1|=|D_2|$, and $l>0$.

The velocity field can be computed through the Biot--Savart law in $\T\times\R$. For that, one must find the Green function associated to the $-\Delta$ operator in order to have an expression for the stream function $\psi$. Later, we just use that $v=\nabla^\perp \psi$, or with the complex notation, $v=2i\partial_{\overline{z}} \psi$. This will be developed in the next result, obtaining different expressions for the velocity, which will be useful later.

\begin{pro}\label{Prop-velocity}
The velocity field of the Euler equations associated to \eqref{FAPointVortex} is given by the following expressions:
\begin{enumerate}
\item \begin{align*}
v_0(x)=&-\frac{1}{2\pi^2}{\int_{\partial D_1}\ln\left|\sin\left(\frac{\pi(x-\xi)}{l}\right)\right|\, d\xi}+\frac{1}{2\pi^2}{\int_{\partial D_2}\ln\left|\sin\left(\frac{\pi(x-\xi)}{l}\right)\right|\, d\xi}.
\end{align*}
\item\label{v3} $$v_0(x)=\frac{i}{2l\pi}\overline{\int_{D_1}\cot\left[\frac{\pi(x-y)}{l}\right]\, dA(y)}-\frac{i}{2l\pi}\overline{\int_{D_2}\cot\left[\frac{\pi(x-y)}{l}\right]\, dA(y)}.$$
\item\label{v4} \begin{align*}
v_0(x)=&\frac{1}{4\pi^2}\overline{{\int_{\partial D_1}\frac{\overline{x}-\overline{\xi}}{x-\xi} \, d\xi}}-\frac{1}{2\pi^2}{\int_{\partial D_1}\ln\left|H\left(\frac{\pi(x-\xi)}{l}\right)\right| \, d\xi}\nonumber\\
&-\frac{1}{4\pi^2}\overline{{\int_{\partial D_2}\frac{\overline{x}-\overline{\xi}}{x-\xi} \, d\xi}}-\frac{1}{2\pi^2}{\int_{\partial D_2}\ln\left|H\left(\frac{\pi(x-\xi)}{l}\right)\right| \, d\xi},
\end{align*}
with
\begin{equation}\label{H}
H(z)=1+\sum_{k\geq 1}\frac{(-1)^k}{(2k+1)!}z^{2k}=\frac{\sin(z)}{z}.
\end{equation}
\end{enumerate}
\end{pro}
\begin{proof}
\medskip
\noindent
(1) Let us begin finding the stream function associated to 
$$
\omega_{0,K}(x_1,x_2)=\frac{1}{\pi}\sum_{|k|\leq K}{\bf{1}}_{D_1+kl}(x_1,x_2)-\frac{1}{\pi}\sum_{|k|\leq K}{\bf{1}}_{D_2+kl}(x_1,x_2),
$$
by superposing the stream function of each one of the elements of the sum, i.e.,
 \begin{equation}\label{psiK}
\psi_{0,K}(x)=\frac{1}{2\pi^2}\sum_{|k|\leq K}\int_{D_1}\ln|x-y-kl|\, dA(y)-\frac{1}{2\pi^2}\sum_{|k|\leq K}\int_{D_2}\ln|x-y-kl|\, dA(y).
 \end{equation}
Using the same idea than in Proposition \ref{Prop-PV}, we find that
$$
\sum_{|k|\leq K}\ln|x-y-kl|=\ln\left|\frac{\pi(x-y)}{l}\prod_{k=1}^{K}\left(1-\frac{(x-y)^2}{k^2l^2}\right)\right|+\ln\left|\frac{l}{\pi}\prod_{k=1}^{K}k^2l^2\right|,
$$
and hence
\begin{align*}
\psi_{0,K}(x)=&\frac{1}{2\pi^2}\int_{D_1}\ln\left|\frac{\pi(x-y)}{l}\prod_{k=1}^{K}\left(1-\frac{(x-y)^2}{k^2l^2}\right)\right|\, dA(y)+\frac{1}{2\pi^2}\ln\left|\frac{l}{\pi}\prod_{k=1}^{K}k^2l^2\right||D_1|\\
&-\frac{1}{2\pi^2}\int_{D_2}\ln\left|\frac{\pi(x-y)}{l}\prod_{k=1}^{K}\left(1-\frac{(x-y)^2}{k^2l^2}\right)\right|\, dA(y)-\frac{1}{2\pi^2}\ln\left|\frac{l}{\pi}\prod_{k=1}^{K}k^2l^2\right||D_2|.
\end{align*}
Using that $D_1$ and $D_2$ have same area, it follows that
\begin{align*}
\psi_{0,K}(x)=&\frac{1}{2\pi^2}\int_{D_1}\ln\left|\frac{\pi(x-y)}{l}\prod_{k=1}^{K}\left(1-\frac{(x-y)^2}{k^2l^2}\right)\right|\, dA(y)\\
&-\frac{1}{2\pi^2}\int_{D_2}\ln\left|\frac{\pi(x-y)}{l}\prod_{k=1}^{K}\left(1-\frac{(x-y)^2}{k^2l^2}\right)\right|\, dA(y),
\end{align*}
where the sine formula \eqref{sine} yields
$$
\psi_0(x)=\frac{1}{2\pi^2}{\int_{D_1}\ln\left|\sin\left(\frac{\pi(x-y)}{l}\right)\right|\, dA(y)}-\frac{1}{2\pi^2}{\int_{D_2}\ln\left|\sin\left(\frac{\pi(x-y)}{l}\right)\right|\, dA(y)}.
$$
Then,
\begin{align}\label{velocity}
v_0(x)=&\frac{i\partial_{\overline{x}}}{\pi^2}{\int_{D_1}\ln\left|\sin\left(\frac{\pi(x-y)}{l}\right)\right|\, dA(y)}-\frac{i\partial_{\overline{x}}}{\pi^2}{\int_{D_2}\ln\left|\sin\left(\frac{\pi(x-y)}{l}\right)\right|\, dA(y)}\nonumber\\
=&-\frac{1}{\pi^2}{\int_{D_1}i\partial_{\overline{y}}\ln\left|\sin\left(\frac{\pi(x-y)}{l}\right)\right|\, dA(y)}+\frac{1}{\pi^2}{\int_{D_2}i\partial_{\overline{y}}\ln\left|\sin\left(\frac{\pi(x-y)}{l}\right)\right|\, dA(y)}\nonumber\\
=&-\frac{1}{2\pi^2}{\int_{\partial D_1}\ln\left|\sin\left(\frac{\pi(x-\xi)}{l}\right)\right|\, d\xi}+\frac{1}{2\pi^2}{\int_{\partial D_2}\ln\left|\sin\left(\frac{\pi(x-\xi)}{l}\right)\right|\, d\xi}.
\end{align}
The Stokes theorem, see Appendix \ref{Ap-potentialtheory}, has been applied in the last line.

\medskip
\noindent
(2) This expression comes from \eqref{velocity} and
$$
2\partial_{\overline{x}} \ln |\sin(x)|=\overline{\cot x},
$$
used in Proposition \ref{Prop-PV}.

\medskip
\noindent
(3) From (1), we can use the series expansion of the complex sine,
$$
\sin(z)=zH(z),\quad H(z)=1+\sum_{k\geq 1}\frac{(-1)^k}{(2k+1)!}z^{2k},
$$
in order to obtain
\begin{align*}
\ln\left|\sin\left(\frac{\pi(x-\xi)}{l}\right)\right|&=\ln\left|\frac{\pi(x-\xi)}{l}\right|+\ln\left|1+\sum_{k\geq 1}\frac{(-1)^k}{(2k+1)!}\frac{\pi^{2k}}{l^{2k}}(x-\xi)^{2k}\right|\\
&=\ln\left|\frac{\pi(x-\xi)}{l}\right|+\ln\left|H\left(\frac{\pi(x-\xi)}{l}\right)\right|.
\end{align*}
Then, we have
\begin{align*}
v_0(x)=&-\frac{1}{2\pi^2}{\int_{\partial D_1}\ln\left|\frac{\pi(x-\xi)}{l}\right| \, d\xi}-\frac{1}{2\pi^2}{\int_{\partial D_1}\ln\left|H\left(\frac{\pi(x-\xi)}{l}\right)\right| \, d\xi}\\
&+\frac{1}{2\pi^2}{\int_{\partial D_2}\ln\left|\frac{\pi(x-\xi)}{l}\right| \, d\xi}+\frac{1}{2\pi^2}{\int_{\partial D_2}\ln\left|H\left(\frac{\pi(x-\xi)}{l}\right)\right| \, d\xi}.
\end{align*}
Moreover, the Stokes formula \eqref{Stokes} yields
\begin{align*}
v_0(x)=&\frac{i}{2\pi^2}\overline{\int_{D_1}\frac{1}{x-y}\, dA(y)}-\frac{1}{2\pi^2}{\int_{\partial D_1}\ln\left|H\left(\frac{\pi(x-\xi)}{l}\right)\right| \, d\xi}\\
&-\frac{i}{2\pi^2}\overline{\int_{D_2}\frac{1}{x-y}\, dA(y)}+\frac{1}{2\pi^2}{\int_{\partial D_2}\ln\left|H\left(\frac{\pi(x-\xi)}{l}\right)\right| \, d\xi}.
\end{align*}
Finally, let us now use the Cauchy--Pompeiu's formula \eqref{Cauchy-Pom} for the first and third terms, to find
 \begin{align*}
v_0(x)=&\frac{1}{4\pi^2}\overline{{\int_{\partial D_1}\frac{\overline{x}-\overline{\xi}}{x-\xi} \, d\xi}}-\frac{1}{2\pi^2}{\int_{\partial D_1}\ln\left|H\left(\frac{\pi(x-\xi)}{l}\right)\right| \, d\xi}\\
&-\frac{1}{4\pi^2}\overline{{\int_{\partial D_2}\frac{\overline{x}-\overline{\xi}}{x-\xi} \, d\xi}}-\frac{1}{2\pi^2}{\int_{\partial D_2}\ln\left|H\left(\frac{\pi(x-\xi)}{l}\right)\right| \, d\xi}.
\end{align*}
\end{proof}

\subsection{Functional setting of the problem}
The first step is to scale the vorticity \eqref{FAPointVortex} in order to introduce the point vortices in our formulation and be able to desingularize them. Let us define
\begin{equation}\label{omega_epsilon}
\omega_{0,\varepsilon}(x)=\frac{1}{\pi\varepsilon^2}\sum_{k\in\Z}{\bf{1}}_{\varepsilon D_1+kl}(x)-\frac{1}{\pi\varepsilon^2}\sum_{k\in\Z}{\bf{1}}_{\varepsilon D_2+kl}(x),
\end{equation}
for $l>0$ and $\varepsilon>0$. The domains $D_1$ and $D_2$ are simply--connected and bounded. In the case that $|D_1|=|\D|$  and $D_2=-D_1+a-ih$, with $a\in\R$ and $h\neq 0$, we find the point vortex street \eqref{PointVortex} passing to the limit $\varepsilon\rightarrow 0$:
\begin{equation}\label{omega00}
\omega_{0,0}(x)=\sum_{k\in\Z}\delta_{(kl,0)}(x)-\sum_{k\in\Z}\delta_{(a+kl,-h)}(x).
\end{equation}
Proposition \ref{Prop-PV} deals with \eqref{omega00} from the dynamical system point of view, showing that it translates. Moreover, if $a=0$ or $a=\frac{l}{2}$, the translation is horizontal.

Now, we try to find the equation that characterize a translating evolution in the Euler equations. Assume that we have  $\omega(t,x)=\omega_0(x-Vt)$, with $V\in\C$. Inserting this ansatz in the Euler equations, we arrived at
$$
\left(v_0(x)-V\right)\cdot \nabla\omega_0(x)=0,\quad x\in\R^2,
$$
where ``$\cdot$'' indicates the scalar product in $\R^2$. As in the previous section, we want to work always in the complex sense identifying $(x_1,x_2)\in\R^2$ as $x_1+ix_2\in\C$. The gradient operator can be identify to the $\partial_{\overline{z}}$ derivative as in \eqref{gradient-complex}. Then, we above equation can be written as:
$$
\textnormal{Re}\left[\overline{(v_0(x)-V)}\partial_{\overline{x}}\omega_0(x)\right]=0, \quad x\in\C.
$$
When working with the scaled vorticity $\omega_{0,\varepsilon}$, this equation must be understood in the weak sense, yielding
\begin{equation}\label{eq1}
\left(v_{0,\varepsilon}(x)-V\right)\cdot \vec{n}(x)=0,\quad x\in\partial\, (\varepsilon D_1+kl)\cup \partial\,  (\varepsilon D_2+kl),
\end{equation}
or similarly,
$$
\textnormal{Re}\left[\overline{(v_{0,\varepsilon}(x)-V)}\vec{n}(x)\right]=0, \quad x\in\partial\, (\varepsilon D_1+kl)\cup \partial\,  (\varepsilon D_2+kl),
$$
for any $k\in\Z$. Here, $\vec{n}$ is the exterior normal vector and $v_{0,\varepsilon}$ is the velocity associated to \eqref{omega_epsilon}. The expression of $v_{0,\varepsilon}$ coming from Proposition \ref{Prop-velocity}--\eqref{v3} gives us
$$
v_{0,\varepsilon}(x)=\frac{i}{2l\pi \varepsilon^2}\overline{\int_{\varepsilon D_1}\cot\left[\frac{\pi(x-y)}{l}\right]\, dA(y)}-\frac{i}{2l \pi\varepsilon^2}\overline{\int_{\varepsilon D_2}\cot\left[\frac{\pi(x-y)}{l}\right]\, dA(y)}.
$$
We can check that $v_{0,\varepsilon}(x+kl)=v_{0,\varepsilon}(x)$, for any $k\in\Z$. Moreover, we have that $\vec{n}_{D+kl}(x+kl)=\vec{n}_{D}(x)$, for any simply--connected bounded domain $D$. Then, the equation \eqref{eq1} reduces to
\begin{align*}
\textnormal{Re}\left[\overline{(v_{0,\varepsilon}(x)-V)}\vec{n}(x)\right]=0,\quad x\in\partial\, D_1\cup \partial\,   D_2.
\end{align*}

Consider $D_2=-D_1+a-ih$, for $a=0$ or $a=\frac{l}{2}$, and $h\neq 0$. By using the relation between $D_1$ and $D_2$, the above system reduces to just one equation:
\begin{align*}
\textnormal{Re}\left[\overline{(v_{0,\varepsilon}(x)-V)}\vec{n}(x)\right]=0,\quad x\in\partial\, D_1,
\end{align*}
where
$$
v_{0,\varepsilon}(\varepsilon x)=\frac{i}{2l\pi }\overline{\int_{ D_1}\cot\left[\frac{\pi\varepsilon(x-y)}{l}\right]\, dA(y)}-\frac{i}{2l \pi}\overline{\int_{ D_1}\cot\left[\frac{\pi(\varepsilon(x+y)-a+ih)}{l}\right]\, dA(y)}.
$$
Other representations of the velocity field can be obtained using Proposition \ref{Prop-velocity}:
$$
v_{0,\varepsilon}(\varepsilon x)=-\frac{1}{2\pi^2\varepsilon}{\int_{\partial D_1}\ln\left|\sin\left(\frac{\pi(\varepsilon(x-\xi))}{l}\right)\right|\, d\xi}-\frac{1}{2\pi^2\varepsilon}{\int_{\partial D_1}\ln\left|\sin\left(\frac{\pi(\varepsilon(x+\xi)-a+ih)}{l}\right)\right|\, d\xi}.
$$
At this stage, we are going to introduce an exterior conformal map from $\T$ into $\partial D_1$ given by
\begin{equation}\label{phi-euler}
\phi(w)=i (w+\varepsilon f(w)), \quad f(w)=\sum_{n\geq 1}a_nw^{-n}, \quad a_n\in\R, w\in\T.
\end{equation}
For others values of $a$, one must readjust the conformal map, but here we will consider $a=0$ and $a=\frac{l}{2}$ having a horizontal translation in the point vortex system.

Note that $\vec{n}(\phi(w))=w\phi'(w)$. Then, we can rewrite the equation with the use of the above conformal map in the following way:
\begin{equation}\label{F-euler}
F_{E}(\varepsilon,f,V)(w):=\textnormal{Re}\left[\left\{\overline{I_E(\varepsilon,f)(w)}-\overline{V}\right\}{w}{\phi'(w)}\right]=0, \quad w\in\T,
\end{equation}
where
\begin{align}\label{Iepsilon}
I_E(\varepsilon,f)(w):=&-\frac{1}{2\pi^2\varepsilon}{\int_{\T}\ln\left|\sin\left(\frac{\pi(\varepsilon(\phi(w)-\phi(\xi)))}{l}\right)\right|\phi'(\xi)\, d\xi}\nonumber\\
&-\frac{1}{2\pi^2\varepsilon}{\int_{\T}\ln\left|\sin\left(\frac{\pi(\varepsilon(\phi(w)+\phi(\xi))-a+ih)}{l}\right)\right|\phi'(\xi)\, d\xi}.
\end{align}
Note that
$$
I_E(\varepsilon,f)(w)=v_{0,\varepsilon}(\varepsilon\phi(w)).
$$
As it is mentioned in the introduction, \cite{HmidiMateu-pairs} deals with the desingularization of a vortex pairs in both the Euler equations and the generalized quasi-geostrophic equation. In order to relate $F_E$ with the functional in \cite{HmidiMateu-pairs}, we can write $I_E$ as
\begin{align*}
I_E(\varepsilon,f)(w)=&\frac{1}{4\pi^2\varepsilon}\overline{{\int_{\T}\frac{\overline{\phi(w)}-\overline{\phi(\xi)}}{\phi(w)-\phi(\xi)}\phi'(\xi) \, d\xi}}-\frac{1}{2\pi^2 \varepsilon}{\int_{\T}\ln\left|H\left(\frac{\pi\varepsilon(\phi(w)-\phi(\xi))}{l}\right)\right| \phi'(\xi)\, d\xi}\nonumber\\
&+\frac{1}{4\pi^2\varepsilon}\overline{{\int_{\T}\frac{\varepsilon(\overline{\phi(w)}+\overline{\phi(\xi)})-a+ih}{\varepsilon(\phi(w)-\phi(\xi))-a+ih}\phi'(\xi) \, d\xi}}\\
&-\frac{1}{2\pi^2 \varepsilon}{\int_{\T}\ln\left|H\left(\frac{\pi(\varepsilon(\phi(w)+\phi(\xi))-a+ih)}{l}\right)\right| \phi'(\xi)\, d\xi}\nonumber\\
=&\frac{1}{4\pi^2\varepsilon}\overline{{\int_{\T}\frac{\overline{\phi(w)}-\overline{\phi(\xi)}}{\phi(w)-\phi(\xi)}\phi'(\xi) \, d\xi}}+\frac{1}{4\pi^2\varepsilon}\overline{{\int_{\T}\frac{\varepsilon(\overline{\phi(w)}+\overline{\phi(\xi)})-a+ih}{\varepsilon(\phi(w)-\phi(\xi))-a+ih}\phi'(\xi) \, d\xi}}\\
&+\tilde{I}_E(\varepsilon,f)(w),
\end{align*}
using the expression of the velocity written in Proposition \ref{Prop-velocity}-(3). The Residue Theorem amounts to
\begin{align}\label{exp-I}
I_E(\varepsilon,f)(w)=&\frac{1}{4\pi^2\varepsilon}\overline{{\int_{\T}\frac{\overline{\phi(w)}-\overline{\phi(\xi)}}{\phi(w)-\phi(\xi)}\phi'(\xi) \, d\xi}\nonumber}\\
&+\frac{1}{4\pi^2}\overline{{\int_{\T}\frac{\overline{\phi(\xi)}}{\varepsilon(\phi(w)-\phi(\xi))-a+ih}\phi'(\xi) \, d\xi}}+\tilde{I}_E(\varepsilon,f)(w)\nonumber\\
=&\hat{I}_E(\varepsilon,f)(w)+\tilde{I}_E(\varepsilon,f)(w).
\end{align}
Hence, the function $\hat{I}_E$ comes from the study of the vortex pairs in \cite{HmidiMateu-pairs}. We will take advantage of the study done in that work about $\hat{I}_E$. In this model, $\hat{I}_{E}$ indicates the contribution of just two vortex patches.

The first step is to check that we recover the point vortex street with this model. Remind that $a=0$ or $a=\frac{l}{2}$.
\begin{pro}\label{Prop-trivialsol}
For any  $h\neq 0$, $l>0$, the following equation is verified
$$
F_E(0, 0, V_0)(w)=0, \quad w\in\T,
$$
where $V_0$ is given by \eqref{velocity_PV}.
\end{pro}
\begin{proof}
The equation that we must check is
\begin{align*}
\lim_{\varepsilon\rightarrow 0}\, \textnormal{Re}\Big[\Big\{&\frac{i}{2\pi^2\varepsilon}\overline{\int_{\T}\ln\left|\sin\left(\frac{\pi\varepsilon i(w-\xi)}{l}\right)\right|\, d\xi}\\
&+\frac{i}{2\pi^2\varepsilon}\overline{\int_{\T}\ln\left|\sin\left(\frac{\pi(\varepsilon i(w+\xi)-a+ih)}{l}\right)\right|\, d\xi}-\overline{V_0}\Big\}{w}i\Big]=0.
\end{align*}
Using the Stokes Theorem, it agrees with
\begin{align*}
\lim_{\varepsilon\rightarrow 0}\, \textnormal{Re}\Big[\Big\{&-\frac{i}{2l\pi }{\int_{ \D}\cot\left[\frac{\pi\varepsilon i(w-y)}{l}\right]\, dA(y)}\\
&+\frac{i}{2l\pi }{\int_{ \D}\cot\left[\frac{\pi(\varepsilon i(w+y)-a+ih)}{l}\right]\, dA(y)}-\overline{V_0}\Big\}{w}i\Big]=0.
\end{align*}
We study the equation in two parts. First, note that
\begin{align*}
\lim_{\varepsilon\rightarrow 0}\, &\textnormal{Re}\left[\left\{\frac{i}{2l \pi}{\int_{ \D}\cot\left[\frac{\pi(\varepsilon i(w+y)-a+ih)}{l}\right]\, dA(y)}-\overline{V_0}\right\}{w}i\right]\\
&=\textnormal{Re}\left[\left\{-\frac{1}{2l\pi i}{\cot\left[\frac{\pi(ih-a)}{l}\right]}|\D|-\overline{V_0}\right\}{w}i\right]\\
&=0, \quad w\in\T.
\end{align*}
In the above limit, we may use the Dominated Convergence Theorem in order to introduce the limit inside the integral. Second, we use the expansion of the complex cotangent as
$$
\cot (z)=\frac{1}{z}+ z T(z),\quad T(z)=\sum_{k=1}^{\infty}\frac{2}{z^2-\pi^2k^2},
$$
where $T$ is a smooth function for $|z|<1$. Then, the only contribution in $F_E$ is given by the first part:
\begin{align*}
\lim_{\varepsilon\rightarrow 0}\, \textnormal{Re}\left[\frac{i}{2l\pi }{\int_{ \D}\cot\left[\frac{\pi\varepsilon i(w-y)}{l}\right]\, dA(y)}{w}i\right]&=\lim_{\varepsilon\rightarrow 0}\,\frac{1}{\varepsilon}\textnormal{Re}\left[\frac{i}{2\pi^2 }{\int_{ \D}\frac{1}{w-y} \, dA(y)}{w}\right]\\
&=\lim_{\varepsilon\rightarrow 0}\,\frac{1}{2\pi\varepsilon}\textnormal{Re}\left[i\right]\\
&=0,
\end{align*}
for $w\in\T$, where we have used the Residue Theorem to compute the integral.
\end{proof}

We fix the Banach spaces that we will use when we apply the Implicit Function Theorem. For $\alpha\in(0,1)$, we define
\begin{align}
X_{\alpha}&=\left\{f\in C^{1,\alpha}(\T), \quad f(w)=\sum_{n\geq 1}a_nw^{-n},\,  a_n\in\R\right\},\label{X}\\
Y_{\alpha}&=\left\{f\in C^{0,\alpha}(\T), \quad f(e^{i\theta})=\sum_{n\geq 2}a_n\sin(n\theta)\label{Y},\,  a_n\in\R\right\}.
\end{align}

\begin{rem}
{{ Let us explain why we need that the first frequency in the domain $Y_{\alpha}$ is vanishing. In the case that $a=0$ or $a=\frac{l}{2}$, it can be checked that $F_E(\varepsilon,f,V)$ is well--defined and $C^1$ from $\R\times X_{\alpha}\times \R$ to
$$
\tilde{Y}_{\alpha}=\left\{f\in C^{0,\alpha}(\T), \quad f(e^{i\theta})=\sum_{n\geq 1}a_n\sin(n\theta),\,  a_n\in\R\right\}.
$$
But, when we linerarize ${F}_E$ and obtain $\partial_f {F}_E(0, 0, V)$, this is not an isomorphism from $X_{\alpha}$ to $\tilde{Y}_{\alpha}$. However, it does from $X_{\alpha}$ to $Y_{\alpha}$. We are using $Y_{\alpha}$ instead of $\tilde{Y}_{\alpha}$ in order to implement later the Implicit Function Theorem.
}}
\end{rem}
\begin{rem}
Note that if $f\in B_{X_{\alpha}}(0,\sigma)$, with $\sigma<1$, then $\phi$ is bilipschitz.
\end{rem}

\begin{pro}\label{Prop-trivialsol2}
The function $V:(-\varepsilon_0,\varepsilon_0)\times B_{X_{\alpha}}(0,\sigma)\longrightarrow \R$, given by 
\begin{align}\label{function-V}
V(\varepsilon,f)=&\frac{\int_{\T}\overline{ I_E(\varepsilon,f)(w)}{w}{\phi'(w)}(1-\overline{w}^2)dw}{\int_{\T} {w}{\phi'(w)}(1-\overline{w}^2)dw},
\end{align}
fulfills $V(0, f)=V_0$, where $V_0$ is defined in \eqref{velocity_PV}. The parameters satisfy: $\varepsilon_0\in(0,\textnormal{min}\{1,\frac{l}{4}\})$, $\sigma<1$, $\alpha\in(0,1)$, and $X$ is defined in \eqref{X}.
\end{pro}
\begin{proof}
In the expression \eqref{function-V}, let us work with the denominator. The Residue Theorem amounts to
$$
\lim_{\varepsilon\rightarrow 0}\int_{\T}w\phi'(w)(1-\overline{w}^2)dw=i\int_{\T} w(1-\overline{w}^2)dw=2\pi.
$$ 
From  \eqref{Iepsilon} and the ideas in Proposition \ref{Prop-trivialsol}, we get
\begin{align}\label{lim-Ie}
\lim_{\varepsilon\rightarrow 0} I_E(\varepsilon,f)(w)=&\lim_{\varepsilon\rightarrow 0}\left\{-\frac{i}{2\pi^2\varepsilon}{\int_{\T}\ln\left|\sin\left(\frac{\pi(\varepsilon i(w-\xi))}{l}\right)\right|\, d\xi}+V_0\right\}\nonumber\\
=&\lim_{\varepsilon\rightarrow 0}\left\{\frac{1}{2\pi^2\varepsilon }\overline{\int_{ \D}\frac{1}{w-y} \, dA(y)}+V_0\right\}\nonumber\\
=&\lim_{\varepsilon\rightarrow 0}\left\{\frac{w}{2\pi\varepsilon}+V_0\right\}.
\end{align}
Note also that
$$
\int_{\T}\overline{w}w(1-\overline{w}^2)dw=\int_{\T}(1-\overline{w}^2)dw=0,
$$
via again the Residue Theorem. Then, the first term in \eqref{lim-Ie} does not provide any contribution. It implies that
$$
V(0,f)=V_0\frac{\int_{\T} {w}(1-\overline{w}^2)dw}{\int_{\T} {w}(1-\overline{w}^2)dw}=V_0.
$$
\end{proof}

\begin{pro}\label{Prop-regEuler}
If $V$ sets \eqref{function-V}, then $$\tilde{F}_E:(-\varepsilon_0,\varepsilon_0)\times B_{X_{\alpha}}(0,\sigma)\rightarrow Y_{\alpha},$$ with $\tilde{F}_E(\varepsilon,f)=F_E(\varepsilon,f,V(\varepsilon,f))$, is well--defined and $C^1$. The parameters satisfy that $\alpha\in(0,1)$, $\varepsilon_0\in(0,\textnormal{min}\{1,\frac{l}{4}\})$ and $\sigma<1$.
\end{pro}
\begin{rem}\label{rem-taylor}
Let us clarify why we need the condition $\varepsilon_0<\frac{l}{4}$. In some point of the proof we need to use Taylor formula in the following way
\begin{align}\label{Taylor-formula}
G(|z_1+z_2|)=G(|z_1|)+\int_0^1G'(|z_1+tz_2|)\frac{\textnormal{Re}\left[(z_1+tz_2)\overline{z_2}\right]}{|z_1+tz_2|}dt,
\end{align}
for $z_1,z_2\in\C$ and { $|z_2|<|z_1|$}. Here, the use of this formula is not explicit since we are refering to the work \cite{HmidiMateu-pairs}, and this condition is needed in order to check $|z_2|<|z_1|$. Although we are not using it explicitly in this proof, we will use it for the general equation in the following section.
\end{rem}
\begin{proof}
We will divide the proof in three steps.

\medskip
\noindent
{\it $\bullet$ First step: Symmetry of $F_E$.} 
Note that $\phi$ given by \eqref{phi-euler} verifies 
$$
\phi(\overline{w})=-\overline{\phi(w)},
$$
where we are taking $\vartheta=i$ in order to work with $a=0$ or $a=\frac{l}{2}$. We are going to check that $F_E(\varepsilon,f,V)(e^{i\theta})=\sum_{n\geq 1}f_n\sin(n\theta)$, with $f_n\in\R$. To do that, it is enough to prove that
$$
F_E(\varepsilon,f,V)(\overline{w})=-F_E(\varepsilon,f,V)(w).
$$
Recall the following property of the complex integrals over $\T$:
\begin{equation}\label{property-integral}
\overline{\int_{\T}f(w)dw}=-\int_{\T}\overline{f(\overline{w})}dw,
\end{equation}
for a complex function $f$.

Let us start with the expression of $I_E(\varepsilon,f)$ and note that
\begin{align*}
\ln\left|\sin\left(\frac{\pi(\varepsilon(\phi(w)+\phi(\xi))-a+ih)}{l}\right)\right|=\ln\left|\sin\left(\frac{\pi(\varepsilon(\phi(w)+\phi(\xi))+ih)}{l}\right)\right|, \quad a=0,\\
\ln\left|\sin\left(\frac{\pi(\varepsilon(\phi(w)+\phi(\xi))-a+ih)}{l}\right)\right|=\ln\left|\cos\left(\frac{\pi(\varepsilon(\phi(w)+\phi(\xi))+ih)}{l}\right)\right|, \quad a=\frac{l}{2}.
\end{align*}
Then,
$$
\ln\left|\sin\left(\frac{\pi(\varepsilon(\phi(\overline{w})+\phi(\overline{\xi}))-a+ih)}{l}\right)\right|=\ln\left|\sin\left(\frac{\pi(\varepsilon(\phi(w)+\phi(\xi))-a+ih)}{l}\right)\right|,
$$
for $a=0$ and $a=\frac{l}{2}$. Notice that $I_E(\varepsilon,f)(\overline{w})=\overline{I_E(\varepsilon,f)(w)}$, which implies
\begin{align*}
-2\pi^2\varepsilon\overline{I_E(\varepsilon,f)(w)}=&\overline{{\int_{\T}\ln\left|\sin\left(\frac{\pi(\varepsilon(\phi(w)-\phi(\xi)))}{l}\right)\right|\phi'(\xi)\, d\xi}}\\
&+\overline{\int_{\T}\ln\left|\sin\left(\frac{\pi(\varepsilon(\phi(w)+\phi(\xi))-a+ih)}{l}\right)\right|\phi'(\xi)\, d\xi}\\
=&-{{\int_{\T}\ln\left|\sin\left(\frac{\pi(\varepsilon(\phi(w)-\phi(\overline{\xi})))}{l}\right)\right|\overline{\phi'(\overline{\xi})}\, d\xi}}\\
&-{\int_{\T}\ln\left|\sin\left(\frac{\pi(\varepsilon(\phi(w)+\phi(\overline{\xi}))-a+ih)}{l}\right)\right|\overline{\phi'(\overline{\xi})}\, d\xi}\\
=&{{\int_{\T}\ln\left|\sin\left(\frac{\pi(\varepsilon(\phi(\overline{w})-\phi(\xi)))}{l}\right)\right|\phi'(\xi)\, d\xi}}\\
&+{\int_{\T}\ln\left|\sin\left(\frac{\pi(\varepsilon(\phi(\overline{w})+\phi(\xi))-a+ih)}{l}\right)\right|\phi'(\xi)\, d\xi}\\
=&-2\pi^2\varepsilon{I_E(\varepsilon,f)(\overline{w})}.
\end{align*}

Next, if $V$ is given by \eqref{function-V}, then we are going to check that $V\in\R$. Let us analyze the denominator and the numerator of the expression of $V$:
\begin{align*}
2i\textnormal{Im}\Big[\int_{\T}\overline{ I_E(\varepsilon,f)(w)}&{w}{\phi'(w)}(1-\overline{w}^2)dw\Big]\\
=&\int_{\T}\overline{ I_E(\varepsilon,f)(w)}{w}{\phi'(w)}(1-\overline{w}^2)dw-\overline{\int_{\T}\overline{ I_E(\varepsilon,f)(w)}{w}{\phi'(w)}(1-\overline{w}^2)dw}\\
=&\int_{\T}\overline{ I_E(\varepsilon,f)(w)}{w}{\phi'(w)}(1-\overline{w}^2)dw+\int_{\T}\overline{ I_E(\varepsilon,f)(w)}{w}\overline{\phi'(\overline{w})}(1-\overline{w}^2)dw\\
=&\int_{\T}\overline{ I_E(\varepsilon,f)(w)}{w}{\phi'(w)}(1-\overline{w}^2)dw-\int_{\T}\overline{ I_E(\varepsilon,f)(w)}{w}{\phi'({w})}(1-\overline{w}^2)dw\\
=&0,
\end{align*}
and
\begin{align*}
2i\textnormal{Im}\left[\int_{\T} {w}{\phi'(w)}(1-\overline{w}^2)dw\right]=&\int_{\T} {w}{\phi'(w)}(1-\overline{w}^2)dw-\overline{\int_{\T} {w}{\phi'(w)}(1-\overline{w}^2)dw}\\
=&\int_{\T} {w}{\phi'(w)}(1-\overline{w}^2)dw+\int_{\T} {w}\overline{\phi'(\overline{w})}(1-\overline{w}^2)dw\\
=&\int_{\T} {w}{\phi'(w)}(1-\overline{w}^2)dw-\int_{\T} {w}{\phi'(w)}(1-\overline{w}^2)dw\\
=&0.
\end{align*}
Then, $V\in\R$. Hence,
\begin{align*}
F_E(\varepsilon,f,V)(\overline{w})=&\textnormal{Re}\left[\left\{\overline{I_E(\varepsilon,f)(\overline{w})-V}\right\}\overline{w}{\phi'(\overline{w})}\right]\\
=&-\textnormal{Re}\left[\left\{{I_E(\varepsilon,f)(w)-V}\right\}{\overline{w}\overline{\phi'(w)}}\right]\\
=&-F_E(\varepsilon,f,V)({w}).
\end{align*}
In order to check that $\tilde{F}_E(\varepsilon,f)\in Y_{\alpha}$, we need $f_1=0$. For that, we ask the condition
$$
\int_0^{2\pi}F_E(\varepsilon,f,V)(e^{i\theta})\sin(\theta)d\theta=-\frac12\int_{\T} F_E(\varepsilon,f,V)(w)(1-\overline{w}^2)dw=0,
$$
which agrees with
$$
\int_{\T} \left\{\overline{I_E(\varepsilon,f)(w)}-V\right\}{w}{\phi'(w)}(1-\overline{w}^2)dw=0.
$$
Using that $V$ verifies \eqref{function-V}, the last equation is clearly set.

\medskip
\noindent
{\it $\bullet$  Second step: Regularity of $V$. }
Let us begin with the denominator, noting that
$$
\int_{\T}w\phi'(w)(1-\overline{w}^2)dw=i\int_{\T}w(1+\varepsilon f'(w))(1-\overline{w}^2)dw=2\pi+i\varepsilon\int_{\T}wf'(w)dw=2\pi-i\varepsilon\int_{\T}f(w)dw,
$$
by using the Residue Theorem. Then, if $|\varepsilon|<\varepsilon_0$ and $f\in B_{X_\alpha}(0,\sigma)$, the denominator is not vanishing. Moreover, the denominator is clearly $C^1$ in $\varepsilon$ and $f$.

We continue with the numerator denoting
\begin{align*}
J(\varepsilon,f)=&\int_{\T}\overline{ I_E(\varepsilon,f)(w)}{w}{\phi'(w)}(1-\overline{w}^2)dw\\
=&\int_{\T}\overline{ \hat{I}_E(\varepsilon,f)(w)}{w}{\phi'(w)}(1-\overline{w}^2)dw+\int_{\T}\overline{ \tilde{I}_E(\varepsilon,f)(w)}{w}{\phi'(w)}(1-\overline{w}^2)dw\\
=:&J_1(\varepsilon,f)(w)+J_2(\varepsilon,f)(w),
\end{align*}
using the decomposition of $I_E$ done in \eqref{exp-I}. Note that $\hat{I}_E$ is the part of $I_E$ coming from the vortex pairs analyzed in \cite{HmidiMateu-pairs}. In that work $J_1$ is analyzed showing that it is $C^1$ in $\varepsilon$ and $f$. Note that the spaces used in \cite{HmidiMateu-pairs} are also \eqref{X}--\eqref{Y} and the condition $\varepsilon_0<\frac{l}{4}$ is needed in their computations, see Remark \ref{rem-taylor}.

Then, it remains to study the regularity of $J_2(\varepsilon,f)$. {We should analyze $\tilde{I}_{E}$, i.e.,
\begin{align*}
\tilde{I}_{E}(\varepsilon,f)(w)=&-\frac{1}{2\pi^2 \varepsilon}{\int_{\T}\ln\left|H\left(\frac{\pi\varepsilon(\phi(w)-\phi(\xi))}{l}\right)\right| \phi'(\xi)\, d\xi}\\
&-\frac{1}{2\pi^2 \varepsilon}{\int_{\T}\ln\left|H\left(\frac{\pi(\varepsilon(\phi(w)+\phi(\xi))-a+ih)}{l}\right)\right| \phi'(\xi)\, d\xi}\\
=&-\frac{i}{2\pi^2 \varepsilon}{\int_{\T}\ln\left|H\left(\frac{\pi\varepsilon(\phi(w)-\phi(\xi))}{l}\right)\right| \, d\xi}\\
&-\frac{i}{2\pi^2}{\int_{\T}\ln\left|H\left(\frac{\pi\varepsilon(\phi(w)-\phi(\xi))}{l}\right)\right| f'(\xi)\, d\xi}\\
&-\frac{i}{2\pi^2 \varepsilon}{\int_{\T}\ln\left|H\left(\frac{\pi(\varepsilon(\phi(w)+\phi(\xi))-a+ih)}{l}\right)\right| \, d\xi}\\
&-\frac{i}{2\pi^2 }{\int_{\T}\ln\left|H\left(\frac{\pi(\varepsilon(\phi(w)+\phi(\xi))-a+ih)}{l}\right)\right|f'(\xi)\, d\xi}\\
=:&-(I_1(\varepsilon,f)+I_2(\varepsilon,f)+I_3(\varepsilon,f)+I_4(\varepsilon,f)).
\end{align*}
Note that $I_2$ and $I_4$ are smooth in both variables, due to that $H(z)=\frac{\sin(z)}{z}$, see \eqref{H}. Then they are $C^1$ in $\varepsilon$ and $\phi$. Let us analyze the others terms. Using \eqref{H} and the expansion of the logarithm,
$$
\ln|1+f(z)|=\textnormal{Re} \sum_1^{\infty} \frac{(-1)^{1+n}f(z)^n}{n},
$$
one has
$$
\ln|H(\varepsilon z)|=\varepsilon^2 G_1(\varepsilon,z),
$$
with $G_1$ smooth in both variables. This implies that $I_1$ is $C^1$ in $\varepsilon$ and $f$. On the other way,
$$
\ln|H(\varepsilon z+z')|=\varepsilon G_2(\varepsilon,z,z')+G_3(z'),
$$
with $G_2$ and $G_3$ smooth. Then, we find
$$
I_3(\varepsilon,f)(w)=\frac{i}{2\pi^2}\int_{\T} G_2\left(\varepsilon, \frac{\pi(\phi(w)+\phi(\xi))}{l},\frac{-a+ih}{l}\right)\, d\xi,
$$
which is smooth in $f$ and $\varepsilon$.
}
Hence, we achieve that $V$ is $C^1$ in both variables.

\medskip
\noindent
{\it $\bullet$  Third step: Regularity of $\tilde{F}_E$. } Decomposing $I$ again as in \eqref{exp-I}, we get that
$$
\tilde{F}(\varepsilon,f)=\textnormal{Re}\left[\left\{\overline{\hat{I}_E(\varepsilon,f)(w)}-\overline{\tilde{I}_E(\varepsilon,f)(w)}-{V}(\varepsilon,f)\right\}{w}{\phi'(w)}\right].
$$
Again, the part coming from $\hat{I}$ is analyzed in \cite{HmidiMateu-pairs}, where it is shown that is $C^1$ in both variables. From the second step, we got that $\tilde{I}$ and $V$ are also smooth completing the proof.
\end{proof}

\subsection{Desingularization of the K\'arm\'an Point Vortex Street}
In this section, we provide the proof of the existence of K\'arm\'an Vortex Patch Street via a desingularization of the point vortex model given by the K\'arm\'an Point Vortex Street. The idea is to implement the Implicit Function Theorem to the functional $\tilde{F}_E$ defined in Proposition \ref{Prop-regEuler}.

\begin{theo}\label{Th-euler}
Let $h, l\in\R$, with $h\neq 0$ and $l>0$, and $a=0$ or $a=\frac{l}{2}$. Then, there exist $D^{\varepsilon}$ such that
\begin{equation}\label{omega_epsilon2}
\omega_0(x)=\frac{1}{\pi\varepsilon^2}\sum_{k\in\Z}{\bf 1}_{\varepsilon D^{\varepsilon}+kl}(x)-\frac{1}{\pi\varepsilon^2}\sum_{k\in\Z}{\bf 1}_{-\varepsilon D^{\varepsilon}+a-ih+kl}(x),
\end{equation}
defines a horizontal translating solution of the Euler equations, with constant speed, for any $\varepsilon\in(0,\varepsilon_0)$ and small enough $\varepsilon_0>0$. Moreover, $D^{\varepsilon}$ is at least $C^1$.
\end{theo}
\begin{proof}
In order to look for solutions in the form \eqref{omega_epsilon2}, we need to study the functional $F_E$ defined in \eqref{F-euler}, where $\phi$ is given by \eqref{phi-euler}. Moreover, $V$ is a function of $(\varepsilon,f)$ described by \eqref{function-V}.

In Proposition \ref{Prop-regEuler}, we have that $\tilde{F}_E:\R\times B_{X_\alpha}(0,\sigma)\rightarrow Y_{\alpha}$, with $\tilde{F}_E(\varepsilon,f)=F_E(\varepsilon,f,V(\varepsilon,f))$, is well--defined and $C^1$, for $\varepsilon_0\in(0,\textnormal{min}\{1,\frac{l}{4}\})$ and $\sigma<1$. Then, we wish to apply the Implicit Function Theorem to $\tilde{F}_E$. By Proposition \ref{Prop-trivialsol} and Proposition \ref{Prop-trivialsol2}, we have that $\tilde{F}_E(0,0)(w)=0$, for any $w\in\T$.

Let us show that $\partial_f \tilde{F}_E(0,0)$ is an isomorphism:
\begin{align*}
\partial_f \tilde{F}_E(0,0)h(w)=\lim_{\varepsilon\rightarrow 0}\textnormal{Re}\Big[&\left\{\partial_f \overline{I_E(0,f)(w)}h(w)-\partial_f V(0, 0)h(w)\right\}iw\\
&+\left\{\overline{I_E(\varepsilon,0)(w)}-V_0\right\}iw\varepsilon h'(w)\Big],
\end{align*}
By Proposition \ref{Prop-trivialsol2}, we obtain $\partial_f V(0, f)h(w)\equiv 0$. Note also that
$$
\lim_{\varepsilon\rightarrow 0} I_E(\varepsilon,0)(w)=\lim_{\varepsilon\rightarrow 0}\left\{\frac{w}{2\pi\varepsilon}+V_0\right\}.$$
Moreover, by expression \eqref{exp-I}, we have
\begin{align*}
\partial_f {I_E(0,0)(w)}h(w)=&\frac{i}{4\pi^2}\overline{\int_{\T}\frac{\overline{h(w)}-\overline{h(\xi)}}{w-\xi}\, d\xi}-\frac{i}{4\pi^2}\overline{\int_{\T}\frac{(h(w)-h(\xi))(\overline{w}-\overline{\xi})}{(w-\xi)^2}\, d\xi}\\
&+\frac{i}{4\pi^2}\overline{\int_{\T}\frac{\overline{w}-\overline{\xi}}{w-\xi}h'(\xi)\, d\xi}+\partial_{f}\tilde{I}_{E}(0, 0)h(w)\\
=&\frac{i}{4\pi^2}\overline{\int_{\T}\frac{\overline{h(w)}-\overline{h(\xi)}}{w-\xi}\, d\xi}-\frac{i}{4\pi^2}\overline{\int_{\T}\frac{(h(w)-h(\xi))(\overline{w}-\overline{\xi})}{(w-\xi)^2}\, d\xi}\\
&+\frac{i}{4\pi^2}\overline{\int_{\T}\frac{\overline{w}-\overline{\xi}}{w-\xi}h'(\xi)\, d\xi}.
\end{align*}
Then, we obtain
\begin{align*}
\partial_f \tilde{F}_E(0,0)h(w)=&\textnormal{Re}\left[\left\{\frac{1}{4\pi^2}\overline{\int_{\T}\frac{\overline{h(w)}-\overline{h(\xi)}}{w-\xi}\, d\xi}-\frac{1}{4\pi^2}\overline{\int_{\T}\frac{(h(w)-h(\xi))(\overline{w}-\overline{\xi})}{(w-\xi)^2}\, d\xi}\right.\right.\\
&\left.\left.+\frac{1}{4\pi^2}\overline{\int_{\T}\frac{\overline{w}-\overline{\xi}}{w-\xi}h'(\xi)\, d\xi}\right\}w+\frac{i}{2\pi}h'(w)\right]\\
=&\textnormal{Im}\left[\left\{\frac{i}{4\pi^2 }\overline{\int_{\T}\frac{\overline{h(w)}-\overline{h(\xi)}}{w-\xi}\, d\xi}-\frac{i}{4\pi^2 }\overline{\int_{\T}\frac{(h(w)-h(\xi))(\overline{w}-\overline{\xi})}{(w-\xi)^2}\, d\xi}\right.\right.\\
&\left.\left.+\frac{i}{4\pi^2 }\overline{\int_{\T}\frac{\overline{w}-\overline{\xi}}{w-\xi}h'(\xi)\, d\xi}\right\}w-\frac{1}{2\pi}h'(w)\right].
\end{align*}
By the Residue Theorem, we have
\begin{align*}
\int_{\T}\frac{\overline{w}-\overline{\xi}}{w-\xi}h'(\xi)\, d\xi=&0,\\
\int_{\T}\frac{\overline{h(w)}-\overline{h(\xi)}}{w-\xi}\, d\xi-\int_{\T}\frac{(h(w)-h(\xi))(\overline{w}-\overline{\xi})}{(w-\xi)^2}\, d\xi=&2i\int_{\T}\frac{\textnormal{Im}\left[\overline{(h(w)-h(\xi))}({w}-{\xi})\right]}{(w-\xi)^2}\, d\xi=0.
\end{align*}
Finally, we find
\begin{align}\label{linop}
\partial_f \tilde{F}_E(0,0)h(w)=-\frac{1}{2\pi}\textnormal{Im}\left[h'(w)\right],
\end{align}
which is an isomorpshim from $X$ to $Y$.
\end{proof}
\begin{rem}
Analyzing \cite{HmidiMateu-pairs}, we realize that the above linearized operator \eqref{linop} agrees with the linearized operator in \cite{HmidiMateu-pairs} for the vortex pairs. This tells us that the only real contribution in the linearized operator is due to two vortex patches: ${\textbf 1}_{\varepsilon D_1}$ and ${\textbf 1}_{-\varepsilon D_1+a-ih}$.
\end{rem}

\subsection{Quasi--geostrophic shallow water equation}
In this section, we investigate the case of the quasi-geostrophic shallow water (QGSW) equations. Let $q$ be the potential vorticity, then the QGSW equations are given by
\begin{eqnarray*}         
       \left\{\begin{array}{ll}
          	q_t+(v\cdot \nabla) q=0, &\text{ in $[0,+\infty)\times\mathbb{R}^2$}, \\
         	 v=\nabla^\perp\psi,&\text{ in $[0,+\infty)\times\mathbb{R}^2$}, \\
         	 \psi=(\Delta-\lambda^2)^{-1}q,&\text{ in $[0,+\infty)\times\mathbb{R}^2$}, \\
         	 q(t=0,x)=q_0(x),& \text{ with $x\in\mathbb{R}^2$},
       \end{array}\right.
\end{eqnarray*}
with $\lambda\neq 0$. The same results to the Euler equations are obtained in this case. That is due to the similarity of the kernel in both cases, in particular, they have the same behavior close to 0. In Section \ref{Sec-NVP} we analyzed the case of the $N$-vortex problem, see Proposition \ref{Prop-PV-QGSW}. Here, we want to desingularize \eqref{PointVortex} in order to obtain periodic in space solutions that translate in the QGSW equation.

The stream function $\psi$ can be recovered in terms of $q$ in the following way
$$
\psi(t,x)=-\frac{1}{2\pi}\int_{\R^2}K_0(|\lambda||x-y|)q(t,y)\, dA(y).
$$
The function $K_0$ is the Modified Bessel function of order zero, whose definition and some of their properties can be found in Appendix \ref{Ap-specialfunctions}. It is of great interest the expansion of $K_0$ given in \eqref{K0-expansion} as
$$
K_0(z)=-\ln\left(\frac{z}{2}\right)I_0(z)+\sum_{k=0}^\infty\frac{\left(\frac{z}{2}\right)^{2k}}{(k!)^2}\varphi(k+1),
$$
where
$$
\varphi(1)=-{\gamma} \quad \text{ and }\quad \varphi(k+1)=\sum_{m=1}^k \frac{1}{m}-{\gamma},\, \, k\in\N^*.
$$
The constant ${\gamma}$ is the Euler's constant and the function $I_0$ is defined in Appendix \ref{Ap-specialfunctions}, but we recall it as
$$
I_0(z)=\sum_{k=0}^\infty\frac{\left(\frac{z}{2}\right)^{2k}}{k!\Gamma(k+1)}.
$$
Via this expansion, one notice that
\begin{equation}\label{expansionK_02}
K_0(z)=-\ln(z)+g_0(z)+g_1,
\end{equation}
where
\begin{align*}
g_0(z)=&-z^2\left(\ln(2)-\ln(z)\right)\sum_{k=1}^{\infty}\frac{\left(\frac{z}{2}\right)^{2k-2}}{k!\Gamma(k+1)}+z^2\sum_{k=1}^{\infty}\frac{\left(\frac{z}{2}\right)^{2k-2}}{(k!)^2}\varphi(k+1),\\
g_1=&-{\gamma}-\ln(2).
\end{align*}
Note that $g_0$ is smooth and $g_0(z)=O(z^2\ln(z))$ close to $0$.

Consider a K\'arm\'an Vortex Patch Street in the QGSW equations in the sense
\begin{align*}
q_0(x)=\frac{1}{	\pi}\sum_{k\in\Z}{\bf{1}}_{D_1+kl}(x)-\frac{1}{	\pi}\sum_{k\in\Z}{\bf{1}}_{D_2+kl}(x),
\end{align*}
where $D_1$ and $D_2$ are simply--connected bounded domains such that $|D_1|=|D_2|$, and $l>0$. Motivated by Euler equations, assume $D_2=-D_1+a-ih$, having the following distribution
\begin{align}\label{FAPointVortex-QGSW}
q_0(x)=\frac{1}{	\pi}\sum_{k\in\Z}{\bf{1}}_{D+kl}(x)-\frac{1}{	\pi}\sum_{k\in\Z}{\bf{1}}_{-D+a+kl-ih}(x),
\end{align}
where we are rewriting $D_1$ by $D$. The velocity field is given by
\begin{align}\label{QGSW-velocity}
v_0(x)=&\frac{\lambda i}{2\pi^2}\sum_{k\in\Z}\int_{D} K_1(\lambda|x-y-kl|)\frac{x-y-kl}{|x-y-kl|}\, dA(y)\nonumber\\
&-\frac{\lambda i}{2\pi^2}\sum_{k\in\Z}\int_{D} K_1(\lambda|x+y-a-kl+ih|)\frac{x+y-a-kl+ih}{|x+y-a-kl+ih|}\, dA(y)\nonumber\\
=&\frac{1}{2\pi^2}\sum_{k\in\Z}\int_{\partial D} K_0(\lambda|x-y-kl|)dy+\frac{1}{2\pi^2}\sum_{k\in\Z}\int_{\partial D} K_0(\lambda|x+y-a-kl+ih|)dy.
\end{align}
We scale the expression \eqref{FAPointVortex-QGSW} in the following way
$$
q_{0,\varepsilon}(x)=\frac{1}{\pi\varepsilon^2}\sum_{k\in\Z}{\bf{1}}_{\varepsilon D+kl}(x)-\frac{1}{\pi\varepsilon^2}\sum_{k\in\Z}{\bf{1}}_{-\varepsilon D+a+kl-ih}(x).
$$
As in the Euler equations, if $|D|=|\D|$, with $a\in\R$ and $h\neq 0$, we obtain the point model:
$$
q_{0,0}(x)=\sum_{k\in\Z}\delta_{(kl,0)}(x)-\sum_{k\in\Z}\delta_{(a+kl,-h)}(x),
$$
which has been studied in Proposition \ref{Prop-PV-QGSW}. Considering now a translating motion in the form $q(t,x)=q_{0}(x-Vt)$, with $V\in\C$, then we arrive at
$$
\left(v_0(x)-V\right)\cdot \nabla q_0(x)=0,\quad x\in\R^2.
$$
In the case of $q_{0,\varepsilon}$, we need to solve the above equation understood in the weak sense, i.e.,
\begin{equation}\label{eq1-QGSW}
\left(v_{0,\varepsilon}(x)-V\right)\cdot \vec{n}(x)=0,\quad x\in\partial\, (\varepsilon D+kl)\cup \partial\,  (-\varepsilon D+a+kl-ih),
\end{equation}
which, as for the Euler equations, reduces to 
\begin{equation}\label{eq2-QGSW}
\textnormal{Re}\left[\overline{\left(v_{0,\varepsilon}(x)-V\right)} \vec{n}(x)\right]=0,\quad x\in\partial\, (\varepsilon D),
\end{equation}
written in the complex sense. With the use of the conformal map \eqref{phi-euler}:
\begin{equation*}
\phi(w)=i (w+\varepsilon f(w)), \quad f(w)=\sum_{n\geq 1}a_nw^{-n}, \quad a_n\in\R, w\in\T,
\end{equation*}
it agrees with
\begin{equation}\label{F-QGSW}
F_{QGSW}(\varepsilon,f,V)(w):=\textnormal{Re}\left[\left\{\overline{I_{QGSW}(\varepsilon,f)(w)}-\overline{V}\right\}{w}{\phi'(w)}\right]=0, \quad w\in\T,
\end{equation}
where
\begin{align}\label{Iepsilon-QGSW}
I_{QGSW}(\varepsilon,f)(w)=v_{0,\varepsilon}(\varepsilon \phi(w)).
\end{align}
Then,
\begin{align*}
I_{QGSW}(\varepsilon,f)(w):=&\frac{1}{2\pi^2 \varepsilon}\sum_{k\in\Z} \int_{\T} K_0(\lambda|\varepsilon(\phi(w)-\phi(\xi))-kl|)\phi'(\xi)\, d\xi\\
&+\frac{1}{2\pi^2 \varepsilon}\sum_{k\in\Z} \int_{\T} K_0(\lambda|\varepsilon(\phi(w)+\phi(\xi))-a-kl+ih|)\phi'(\xi)\, d\xi.
\end{align*}
Via the expansion of $K_0$, given in \eqref{expansionK_02}, one has
\begin{align}\label{decompI}
I_{QGSW}(\varepsilon,f)(w)=&-\frac{1}{2\pi^2 \varepsilon}\sum_{k\in\Z} \int_{\T} \ln\left|\lambda(\varepsilon(\phi(w)-\phi(\xi))-kl)\right|\phi'(\xi)\, d\xi\nonumber\\
&-\frac{1}{2\pi^2 \varepsilon}\sum_{k\in\Z} \int_{\T} \ln\left|(\varepsilon(\phi(w)+\phi(\xi))-a-kl+ih)\right|\phi'(\xi)\, d\xi\nonumber\\
&+\frac{1}{2\pi^2 \varepsilon}\sum_{k\in\Z} \int_{\T} g_0\left(|\varepsilon(\phi(w)-\phi(\xi))-kl|\right)\phi'(\xi)\, d\xi\nonumber\\
&+\frac{1}{2\pi^2 \varepsilon}\sum_{k\in\Z} \int_{\T} g_0\left(\lambda|\varepsilon(\phi(w)+\phi(\xi))-a-kl+ih|\right)\phi'(\xi)\, d\xi\nonumber\\
=&I_{E}(\varepsilon,f)(w)+\tilde{I}_{QGSW}(\varepsilon,f)(w),
\end{align}
where $I_E$ is the corresponding function associated to Euler equations, see \eqref{Iepsilon}.

The analogue to Propositions \ref{Prop-trivialsol}, \ref{Prop-trivialsol2} and \ref{Prop-regEuler}, and Theorem \ref{Th-euler} are obtained, whose proofs are very similar and so here we omit many details. Remark that $a=0$ or $a=\frac{l}{2}$.

\begin{pro}\label{Prop-trivialsolQGSW}
For any  $h\neq 0$ and $l>0$, the following equation is verified
$$
F_{QGSW}(0, 0, V_0)(w)=0, \quad w\in\T,
$$
where $V_0$ is given by \eqref{velocity_PV-QGSW}:
$$
V_0=\frac{\lambda i}{2\pi}\sum_{ k\in\Z}K_1(\lambda|a+kl-ih|)\frac{a+kl-ih}{|a+kl-ih|}.
$$
\end{pro}
{{\begin{proof}
Using definition \eqref{F-QGSW}, we need to check that
\begin{align*}
\lim_{\varepsilon\rightarrow 0}\textnormal{Re}\Big[\Big\{&-\frac{i}{2\pi^2\varepsilon} \sum_{k\in\Z} \overline{\int_{\T}K_0(\lambda|\varepsilon i(w-\xi)-kl|)\, d\xi}\\
&-\frac{i}{2\pi^2\varepsilon} \sum_{k\in\Z} \overline{\int_{\T}K_0(\lambda|\varepsilon i(w+\xi)-a-kl+ih|)\, d\xi}-\overline{V_0}\Big\}w i\Big]=0.
\end{align*}
Via the Stokes Theorem, the expansion of $K_0$ given in \eqref{expansionK_02} and noting that
\begin{equation}\label{property-nabla}
2\partial_{\overline{x}} G(|ix+x'|)=-iG'(|ix+x'|)\frac{ix+x'}{|ix+x'|^2},
\end{equation}
for $x,x'\in\C$ and some function $G:\R\rightarrow\R$, it reduces to 
\begin{align*}
\lim_{\varepsilon\rightarrow 0}& \textnormal{Re}\Big[\Big\{\frac{\lambda}{2\pi^2\varepsilon}  \int_{\D}\frac{\, dA(y)}{w-y}-\frac{i}{2\pi^2\varepsilon}\overline{\int_{\T}g_0(\varepsilon\lambda|w-\xi|)\, d\xi}-\frac{i}{2\pi^2\varepsilon}\sum_{0\neq k\in\Z}\overline{\int_{\T}K_0(\lambda|\varepsilon i(w-\xi)-kl|)\, d\xi}\\
&-\frac{\lambda\varepsilon}{2\pi^2\varepsilon} \sum_{k\in\Z} \overline{i\int_{\D}K_1(\lambda|\varepsilon i(w+y)-a-kl+ih|)\frac{\varepsilon i(w-y)-a-kl+ih}{|\varepsilon i(w-y)-a-kl+ih|} \, dA(y)}-\overline{V_0}\Big\}w i\Big].
\end{align*}
Note that
$$
\lim_{\varepsilon\rightarrow 0} \frac{g_0(\varepsilon\lambda|w-\xi|)}{\varepsilon}=0,
$$
and
\begin{align*}
\lim_{\varepsilon\rightarrow 0}\sum_{0\neq k\in\Z} \frac{i}{\varepsilon}\int_{\T}K_0(\lambda|\varepsilon i(x-\xi)-kl|)\, d\xi=&\lim_{\varepsilon\rightarrow 0}i\lambda\sum_{0\neq k\in\Z}\int_{\D} K_1(\lambda|\varepsilon i(w-\xi)-kl|)\frac{\varepsilon i(w-\xi)-kl}{|\varepsilon i(w-\xi)-kl|^2}\\
=&i\lambda\sum_{0\neq k\in\Z}\int_{\D} K_1(\lambda|kl|)\frac{kl}{|kl|^2}=0,
\end{align*}
making use of the Dominated Convergence Theorem. Secondly, via the definition of $V_0$ in \eqref{velocity_PV-QGSW}, one has that
$$
-\frac{\lambda}{2\pi^2} \sum_{k\in\Z} \overline{i\int_{\D}K_1(\lambda|-a-kl+ih|)\frac{-a-kl+ih}{|-a-kl+ih|} \, dA(y)}-\overline{V_0}=0.
$$
The left term is also zero using the computations in Proposition \ref{Prop-trivialsol}.
\end{proof}}}

We avoid the proof of the following result, due to the similarity with Proposition \ref{Prop-trivialsol2}.
\begin{pro}\label{Prop-trivialsol2QGSW}
The function $V:(-\varepsilon_0,\varepsilon_0)\times B_{X_{\alpha}}(0,\sigma)\longrightarrow \R$, given by 
\begin{align}\label{function-V-QGSW}
V(\varepsilon,f)=\frac{\int_{\T}\overline{ I_{QGSW}(\varepsilon,f)(w)}{w}{\phi'(w)}(1-\overline{w}^2)dw}{\int_{\T} {w}{\phi'(w)}(1-\overline{w}^2)dw},
\end{align}
fulfills $V(0, f)=V_0$, where $V_0$ is defined in \eqref{velocity_PV-QGSW}. The parameters satisfy: $\varepsilon_0\in(0,\textnormal{min}\{1,\frac{l}{4}\})$, $\sigma<1$, $\alpha\in(0,1)$, and $X$ is defined in \eqref{X}.
\end{pro}

The next result concerns the well-definition of $F_{QGSW}$ in the spaces defined in \eqref{X}--\eqref{Y}.

\begin{pro}\label{Prop-regQGSW}
If $V$ sets \eqref{function-V-QGSW}, then $$\tilde{F}_{QGSW}:(-\varepsilon_0,\varepsilon_0)\times B_{X_\alpha}(0,\sigma)\rightarrow Y_\alpha,$$ with $\tilde{F}_{QGSW}(\varepsilon,f)=F_{QGSW}(\varepsilon,f,V(\varepsilon,f))$, is well--defined and $C^1$. The parameters satisfy $\alpha\in(0,1)$, $\varepsilon_0\in(0,\textnormal{min}\{1,\frac{l}{4}\})$ and $\sigma<1$.
\end{pro}
{{\begin{proof}
Note the similarity of this proposition to Proposition \ref{Prop-regEuler}. Following its ideas, in order to check the symmetry of $F_{QGSW}$, it is enough to prove that
$$
I_{QGSW}(\varepsilon,f)(\overline{w})=\overline{I_{QGSW}(\varepsilon,f)(w)}.
$$
We take advantage of \eqref{property-integral}. Via its definition, note that
\begin{align*}
\overline{I_{QGSW}(\varepsilon,f)({w})}=&-\frac{1}{2\pi^2 \varepsilon}\sum_{k\in\Z} \int_{\T} K_0(\lambda|\varepsilon(\phi(w)-\phi(\overline{\xi}))-kl|)\overline{\phi'(\overline{\xi})}\, d\xi\\
&-\frac{1}{2\pi^2 \varepsilon}\sum_{k\in\Z} \int_{\T} K_0(\lambda|\varepsilon(\phi(w)+\phi(\overline{\xi}))-a-kl+ih|)\overline{\phi'(\overline{\xi})}\, d\xi\\
=&\frac{1}{2\pi^2 \varepsilon}\sum_{k\in\Z} \int_{\T} K_0(\lambda|\varepsilon(\phi(\overline{w})-\phi({\xi}))+kl|){\phi'({\xi})}\, d\xi\\
&+\frac{1}{2\pi^2 \varepsilon}\sum_{k\in\Z} \int_{\T} K_0(\lambda|\varepsilon(\phi(\overline{w})+\phi({\xi}))-a+kl+ih|){\phi'({\xi})}\, d\xi\\
=&I_{QGSW}(\varepsilon,f)(\overline{w}).
\end{align*}
Note that using the decomposition of $I_{QGSW}$ given in \eqref{decompI}, the regularity problem reduces to the same one for the Euler equations, done in Proposition \ref{Prop-regEuler}.
\end{proof}}}

Finally, we state the result concerning the desingularization of the K\'arm\'an Vortex Street.
\begin{theo}\label{Th-QGSW}
Let $h, l\in\R$, with $h\neq 0$ and $l>0$, and $a=0$ or $a=\frac{l}{2}$. Then, there exist $D^{\varepsilon}$ such that
\begin{equation}\label{omega_epsilon2-qgsw}
q_{0,\varepsilon}(x)=\frac{1}{\pi\varepsilon^2}\sum_{k\in\Z}{\bf 1}_{\varepsilon D^{\varepsilon}+kl}(x)-\frac{1}{\pi\varepsilon^2}\sum_{k\in\Z}{\bf 1}_{-\varepsilon D^{\varepsilon}+a-ih+kl}(x),
\end{equation}
defines a horizontal translating solution of the quasi--geostrophic shallow water equations, with constant velocity speed, for any $\varepsilon\in(0,\varepsilon_0)$ and small enough $\varepsilon_0>0$. Moreover, $D^{\varepsilon}$ is at least $C^1$.
\end{theo}
{{\begin{proof}
By Proposition \ref{Prop-regQGSW}, we have that $\tilde{F}_{QGSW}:\R\times B_{X_\alpha}(0,\sigma)\rightarrow Y_\alpha$, with $\tilde{F}_{QGSW}(\varepsilon,f)=F_{QGSW}(\varepsilon,f,V(\varepsilon,f))$, is well--defined and $C^1$, for $\varepsilon_0\in(0,\textnormal{min}\{1,\frac{l}{4}\})$ and $\sigma<1$. Moreover, Proposition \ref{Prop-trivialsolQGSW} and Proposition \ref{Prop-trivialsol2QGSW} give us that $\tilde{F}_{QGSW}(0,0)(w)=0$, for any $w\in\T$. In order to apply the Implicit Function Theorem, let us check that $\partial_f \tilde{F}_{QGSW}(0,0)$ is an isomorphism.

First, using \eqref{decompI}, one achieves
\begin{align*}
I_{QGSW}(\varepsilon,f)(w)=&I_{E}(\varepsilon,f)(w)+\frac{1}{2\pi^2 \varepsilon}\sum_{k\in\Z} \int_{\T} g_0\left(|\varepsilon(\phi(w)-\phi(\xi))-kl|\right)\phi'(\xi)\, d\xi\nonumber\\
&+\frac{1}{2\pi^2 \varepsilon}\sum_{k\in\Z} \int_{\T} g_0\left(\lambda|\varepsilon(\phi(w)+\phi(\xi))-a-kl+ih|\right)\phi'(\xi)\, d\xi\nonumber,
\end{align*}
where $g_0$ is a smooth function such that $g_0(z)=O(z^2\ln(z))$ for small $z$, see \eqref{expansionK_02}.
Note that
$$
\lim_{\varepsilon\rightarrow 0}\varepsilon I_{QGSW}(\varepsilon,f)(w)=\lim_{\varepsilon\rightarrow 0}\varepsilon I_{E}(\varepsilon,f)(w)
$$
and
$$
\lim_{\varepsilon\rightarrow 0}\partial_f \overline{I_{QGSW}(\varepsilon,0)}=\lim_{\varepsilon\rightarrow 0}\partial_f \overline{I_{E}(\varepsilon,0)}.
$$
Thus, we have
$$
\partial_f \tilde{F}_{QGSW}(\varepsilon, 0)h(w)=\partial_f \tilde{F}_{E}(\varepsilon,0)h(w)=-\frac{1}{2\pi}\textnormal{Im}\left[h'(w)\right],
$$
which is an isomorphism from $X$ to $Y$.
\end{proof}}}

\section{K\'arm\'an Vortex Street in general models}\label{Sec3}
K\'arm\'an Vortex Patch Street structures are found both in the Euler equations and in the QGSW equations. The important fact in both models is that the Green functions associated to the elliptic problem of the stream function have the same behavior close to 0, having then the same linearized operator. We can extend it to other models, where the generalized surface quasi--geostrophic equations are a particular case, see Theorem \ref{Th-gSQG} for more details. Here, let us work with the general model:
\begin{eqnarray} \label{Generaleq}
       \left\{\begin{array}{ll}
          	q_t+(v\cdot \nabla) q=0, &\text{ in $[0,+\infty)\times\mathbb{R}^2$}, \\
         	 v=\nabla^\perp \psi,&\text{ in $[0,+\infty)\times\mathbb{R}^2$}, \\
         	 \psi=G*q,&\text{ in $[0,+\infty)\times\mathbb{R}^2$}, \\
         	 q(t=0,x)=q_0(x),& \text{ with $x\in\mathbb{R}^2$}.
       \end{array}\right.
\end{eqnarray}

\subsection{Scaling the equation}
The aim of this section is to look for solutions of the type
\begin{align*}
q_0(x)=\frac{1}{\pi}\sum_{k\in\Z}{\bf{1}}_{D_1+kl}(x)-\frac{1}{\pi}\sum_{k\in\Z}{\bf{1}}_{D_2+kl}(x).
\end{align*}
The domains $D_1$ and $D_2$ are simply--connected bounded domains such that $|D_1|=|D_2|$, and $l>0$. Consider $D_2=-D_1+a-ih$, having the following distribution
\begin{align}\label{FAPointVortex-gen}
q_0(x)=\frac{1}{\pi}\sum_{k\in\Z}{\bf{1}}_{D+kl}(x)-\frac{1}{\pi}\sum_{k\in\Z}{\bf{1}}_{-D+a+kl-ih}(x),
\end{align}
where we are rewriting $D_1$ by $D$. The velocity field is given by
\begin{align}\label{gen-velocity}
\pi v_0(x)=&2i\partial_{\overline{x}}\sum_{k\in\Z}\int_{D} G(|x-y-kl|)\, dA(y)-2i\partial_{\overline{x}}\sum_{k\in\Z}\int_{D} G(|x+y-a-kl+ih|)\, dA(y)\nonumber\\
=&-2i\sum_{k\in\Z}\int_{D}\partial_{\overline{y}}G(|x-y-kl|)\, dA(y)-2i\sum_{k\in\Z}\int_{D} \partial_{\overline{y}} G(|x+y-a-kl+ih|)\, dA(y)\nonumber\\
=&-\sum_{k\in\Z}\int_{\partial D} G(|x-y-kl|)dy-\sum_{k\in\Z}\int_{\partial D} G(|x+y-a-kl+ih|)dy,
\end{align}
where $\partial_{\overline{x}}$ is defined in \eqref{gradient-complex} and the Stokes Theorem \eqref{Stokes} is used. In order to introduce the point model configuration, let us scale the equation in the following way. For any $\varepsilon>0$, define
\begin{equation}\label{omega_epsilon-gen}
q_{0,\varepsilon}(x)=\frac{1}{\pi\varepsilon^2}\sum_{k\in\Z}{\bf{1}}_{\varepsilon D+kl}(x)-\frac{1}{\pi\varepsilon^2}\sum_{k\in\Z}{\bf{1}}_{-\varepsilon D+a+kl-ih}(x),
\end{equation}
for $l>0$, $h\neq 0$ and $a\in\R$. If $|D|=|\D|$, then we arrive at the point vortex street \eqref{PointVortex-gen} in the limit when $\varepsilon\rightarrow 0$, i.e.,
\begin{equation}\label{PV-reg}
q_{0,0}(x)=\sum_{k\in\Z}\delta_{(kl,0)}(x)-\sum_{k\in\Z}\delta_{(a+kl,-h)}(x).
\end{equation}
This configuration of points is studied in Proposition \ref{Gen-point}, from the dynamical system point of view, showing that \eqref{PV-reg} translates. From now on, take $a=0$ or $a=\frac{l}{2}$ having a horizontal translation in the point model. The associated velocity field to \eqref{omega_epsilon-gen} is given by
$$
v_{0,\varepsilon}(\varepsilon x)=-\frac{1}{\pi \varepsilon}\sum_{k\in\Z}\int_{\partial D} G(|\varepsilon(x-y)-kl|)dy-\frac{1}{\pi\varepsilon}\sum_{k\in\Z}\int_{\partial D} G(|\varepsilon(x+y)-a-kl+ih|)dy.
$$
We introduce now the conformal map. Consider $\phi:\T\rightarrow \partial D$ such that
\begin{equation}\label{phi-gen}
\phi(w)=i \left(w+\frac{\varepsilon}{G(\varepsilon)} f(w)\right), \quad f(w)=\sum_{n\geq 1}a_nw^{-n}, \quad a_n\in\R, w\in\T.
\end{equation}
Hence
\begin{align*}
v_{0,\varepsilon}(\varepsilon\phi(w))=&-\frac{1}{\pi\varepsilon}\sum_{k\in\Z}\int_{\T} G(|\varepsilon(\phi(w)-\phi(\xi))-kl|)\phi'(\xi)\, d\xi\\
&-\frac{1}{\pi\varepsilon}\sum_{k\in\Z}\int_{\T} G(|\varepsilon(\phi(w)+\phi(\xi))-a-kl+ih|)\phi'(\xi)\, d\xi.
\end{align*}
\begin{rem}
The constant $G(\varepsilon)$ in the definition of the conformal map \eqref{phi-euler} comes from the singularity of the kernel in the general case. For the logarithmic singularities, we do not need to add this constant because there we use the structure of the logarithm. When having more singular kernels, as in this case, we need to introduce $G(\varepsilon)$.
\end{rem}

Assuming that we look for translating solutions, i.e., $q(t,x)=q_0(x- Vt)$, we arrive at the equation
\begin{equation}\label{F-gen}
F(\varepsilon,f,V)(w):=\textnormal{Re}\left[\left\{\overline{I(\varepsilon,f)(w)}- \overline{V}\right\}{w}{\phi'(w)}\right]=0, \quad w\in\T,
\end{equation}
where
$$
I(\varepsilon,f)(w):=v_{0,\varepsilon}(\varepsilon\phi(w)).
$$
The next step is to check that if $\varepsilon=0$, $D=\D$ and $V=V_0$ (referring to the K\'arm\'an Point Vortex Street), equation \eqref{F-gen} is verified.

\begin{pro}\label{Prop-trivialsol-gen}
Let $G$ satisfies
\begin{enumerate}
\item[(H1)] $G$ is radial such that $G(x)=\tilde{G}(|x|)$,
\item[(H2)] there exists $R>0$ and $\beta_1\in(0,1]$ such that $|\tilde{G}'(r)|\leq \frac{C}{r^{1+\beta_1}}$, for $r\geq R$,
\item[(H3)]\label{H3} there exists $\beta_2\in(0,1)$ such that $G(z)=O\left(\frac{1}{z^{\beta_2}}\right)$ and $\log|z|=o\left(G(z)\right),$ as $z\rightarrow 0$. 
\end{enumerate}
For any  $h\neq 0$ and $l>0$, the following equation is verified
$$
F(0, 0, V_0)(w)=0, \quad w\in\T,
$$
where $V_0$ is given by \eqref{V_0-gen}:
$$
V_0=i\sum_{k\in\Z} G'(|a+kl-ih|)\frac{a+kl-ih}{|a+kl-ih|}.
$$
\end{pro}
\begin{rem}
From Hypothesis (H3) we are assuming that the kernel is more singular than the logarithmic kernel analyzed in the previous sections, but less singular than the kernel of the surface quasi--geostrophic equation. A typical kernel satisfying (H3) is the one coming from the generalized surface quasi--geostrophic equation: $G(x)=\frac{C_{\beta}}{2\pi}\frac{1}{|x|^{\beta}}$, for $\beta\in(0,1)$.
\end{rem}
{{\begin{proof}
By definition,
\begin{align*}
F(\varepsilon,0, V_0)(w)=&\textnormal{Re}\Big[\left\{\frac{i}{\pi\varepsilon}\overline{\int_{\T}G(\varepsilon|w-\xi|)\, d\xi}+\frac{i}{\pi\varepsilon}\sum_{0\neq k\in\Z}\overline{\int_{\T} G(|\varepsilon i(w-\xi)-kl|)\, d\xi}\right.\\
&\left.+\frac{i}{\pi\varepsilon}\sum_{k\in\Z}\overline{\int_{\T} G(|\varepsilon i(w+\xi)-a-kl+ih|)\, d\xi}-V_0\right\}iw\Big].
\end{align*}
Concerning the first term, we can compute it using
\begin{align}\label{int-K}
\textnormal{Re}\left[\frac{w}{\pi\varepsilon}\overline{\int_{\T}G(\varepsilon|w-\xi|)\, d\xi}\right]=\textnormal{Re}\left[\frac{w}{\pi\varepsilon}\overline{w}\overline{\int_{\T} G(\varepsilon|1-\xi|)\, d\xi}\right]
=\frac{1}{\pi\varepsilon}\textnormal{Re}\left[\overline{\int_{\T} G(\varepsilon|1-\xi|)\, d\xi}\right]=0,
\end{align}
since the integral is a pure complex number. For the second term, note that
\begin{align}\label{2ndterm-gen}
\lim_{\varepsilon \rightarrow 0}\frac{1}{\varepsilon}\sum_{0\neq k\in\Z}\int_{\T}& G(|\varepsilon i(w-\xi)-kl|)\, d\xi\nonumber\\
=&-\lim_{\varepsilon \rightarrow 0}\sum_{0\neq k\in\Z}\int_{\D}G'(|\varepsilon i(w-\xi)-kl|)\frac{\varepsilon i(w-\xi)-kl}{|\varepsilon i(w-\xi)-kl|}\, d\xi\nonumber\\
=&\sum_{0\neq k\in\Z}\int_{\D}G'(|kl|)\frac{kl}{|kl|}\, d\xi\nonumber\\
=&0,
\end{align}
via the Stokes Theorem \eqref{Stokes} and taking into account \eqref{property-nabla}. We have computed the above limit by using the Convergence Dominated Theorem and (H3). Moreover, the sum is vanishing because we are using the symmetry sum. Using again the Stokes Theorem for the third term, one arrives at
\begin{align*}
\lim_{\varepsilon \rightarrow 0}\left\{\frac{i}{\pi\varepsilon}\right.&\left.\sum_{k\in\Z}\overline{\int_{\T} G(|\varepsilon i(w+\xi)-a-kl+ih|)\, d\xi}-V_0\right\}\\
=&\lim_{\varepsilon \rightarrow 0}\left\{\frac{i}{\pi}\sum_{k\in\Z}\overline{\int_{\D} G'(|\varepsilon i(w+\xi)-a-kl+ih|)\frac{\varepsilon i(w+\xi)-a-kl+ih}{|\varepsilon i(w+\xi)-a-kl+ih|} \, d\xi}-V_0\right\}\\
=&\left\{\frac{i}{\pi}\sum_{k\in\Z}\overline{\int_{\D} G'(|-a-kl+ih|)\frac{-a-kl+ih}{|-a-kl+ih|} \, d\xi}-V_0\right\}\\
=&\left\{i\sum_{k\in\Z}\overline{ G'(|-a-kl+ih|)\frac{-a-kl+ih}{|-a-kl+ih|} \, d\xi}-V_0\right\}\\
=&0,
\end{align*}
by definition of $V_0$, which is given in \eqref{V_0-gen}. Again, the above limit is justified with the Convergence Dominated Theorem and (H3).
\end{proof}}}
In what follows, we will have to deal with the singularity in $\varepsilon$. But, there is a way to simplify the equation in order to control this singularity.  We decompose $I(\varepsilon,f)$ as
\begin{align}\label{I_exp}
-\pi I(\varepsilon,f)(w)=&\frac{1}{\varepsilon}\int_{\T} G(|\varepsilon(\phi(w)-\phi(\xi))|)\phi'(\xi)\, d\xi+\frac{1}{\varepsilon}\sum_{0\neq k\in\Z}\int_{\T} G(|\varepsilon(\phi(w)-\phi(\xi))-kl|)\phi'(\xi)\, d\xi\nonumber\\
&+\frac{1}{\varepsilon}\sum_{k\in\Z}\int_{\T} G(|\varepsilon(\phi(w)+\phi(\xi))-a-kl+ih|)\phi'(\xi)\, d\xi\nonumber\\
=&:I_1(\varepsilon,f)(w)+I_2(\varepsilon,f)(w)+I_3(\varepsilon,f)(w).
\end{align}
We can use Taylor formula \eqref{Taylor-formula}:
\begin{align*}
G(|z_1+z_2|)=G(|z_1|)+\int_0^1G'(|z_1+tz_2|)\frac{\textnormal{Re}\left[(z_1+tz_2)\overline{z_2}\right]}{|z_1+tz_2|}dt,
\end{align*}
for $z_1,z_2\in\C$ and { $|z_2|<|z_1|$}. In the case of $I_1$, take $z_1=i\varepsilon(w-\xi)$ and $z_2=i\frac{\varepsilon^2}{K(\varepsilon)}(f(w)-f(\xi)),$ implying
\begin{align*}
I_1(\varepsilon,f)=&\frac{i}{\varepsilon}\int_{\T}G(\varepsilon|w-\xi|)\, d\xi+i\frac{\varepsilon}{G(\varepsilon)}\int_{\T}\int_0^1G'\left(\varepsilon\left|(w-\xi)+t\frac{\varepsilon}{G(\varepsilon)}(f(w)-f(\xi))\right|\right)\\
&\times\frac{\textnormal{Re}\left[\left((w-\xi)+t\frac{\varepsilon}{G(\varepsilon)}(f(w)-f(\xi))\right)\overline{(f(w)-f(\xi))}\right]}{|(w-\xi)+t\frac{\varepsilon}{G(\varepsilon)}(f(w)-f(\xi))|}dt\, d\xi\\
&+\frac{i}{G(\varepsilon)}\int_{\T}G(\varepsilon|\phi(w)-\phi(\xi)|)f'(\xi)\, d\xi\\
=&I_{1,1}(\varepsilon,f)(w)+I_{1,2}(\varepsilon,f)(w)+I_{1,3}(\varepsilon,f)(w).
\end{align*}
Let us check that $|z_2|<|z_1|$:
$$
|z_2|=\frac{\varepsilon^2}{G(\varepsilon)}|f(w)-f(\xi)|\leq \frac{\varepsilon^2}{G(\varepsilon)}||f||_{C^1}|w-\xi|\leq \frac{\varepsilon^2}{G(\varepsilon)}|w-\xi|< |z_1|.
$$
In virtue of \eqref{int-K}, we get
$$
\textnormal{Re}\left[\frac{w}{\pi\varepsilon}\overline{\int_{\T}G(\varepsilon|w-\xi|)\, d\xi}\right]=0,
$$
and then the nonlinear function $F(\varepsilon,f,V)$ can be simplified as follows
\begin{align}\label{F2}
F(\varepsilon,f,V)=&\textnormal{Re}\left[\left\{\overline{I(\varepsilon,f)(w)}- \overline{V}\right\}{w}{\phi'(w)}\right]-\textnormal{Re}\left[\frac{wf'(w)}{\pi G(\varepsilon)}\overline{\int_{\T}G(\varepsilon|w-\xi|)\, d\xi}\right]\nonumber\\
=&\textnormal{Re}\left[\left\{\overline{I(\varepsilon,f)(w)}- \overline{V}\right\}{w}{\phi'(w)}\right]+\frac{i\int_{\T}G(\varepsilon|1-\xi|)\, d\xi}{\pi G(\varepsilon)}\textnormal{Im}\left[f'(w)\right].
\end{align}
We use the descomposition of $I(\varepsilon,f)$ in \eqref{I_exp}, and we are rewriting $I_1(\varepsilon,f)$ as
\begin{equation}\label{I1-decom}
I_1(\varepsilon,f):= I_{1,2}(\varepsilon,f)+I_{1,3}(\varepsilon,f),
\end{equation}
since there is any contribution of $I_{1,1}(\varepsilon,f)$.

In the next result, we provide how $V$ must depend on $\varepsilon$ and $f$.
\begin{pro}\label{Prop-trivialsol2Qgen}
Let $G$ satisfies the hypothesis (H1)--(H3) of Proposition \ref{Gen-point} and Proposition \ref{Prop-trivialsol-gen}, and
\begin{enumerate}
\item[(H4)] {$\frac{\tilde{G}(\varepsilon r)}{\tilde{G}(\varepsilon)\tilde{G}(r)}\rightarrow 1,\, \frac{\varepsilon\tilde{G}'(\varepsilon r)}{\tilde{G}(\varepsilon)\tilde{G}'(r)}\rightarrow 1,\textnormal{when } \varepsilon\rightarrow 0,
$
uniformly in $r\in(0,2)$.}
\end{enumerate} The function $V:(-\varepsilon_0,\varepsilon_0)\times B_{X_{1-\beta_2}}(0,\sigma)\longrightarrow \R$, given by 
\begin{align}\label{function-V-gen}
V(\varepsilon,f)=&\frac{\int_{\T}\overline{ I(\varepsilon,f)(w)}{w}{\phi'(w)}(1-\overline{w}^2)dw}{\int_{\T} {w}{\phi'(w)}(1-\overline{w}^2)dw},
\end{align}
fulfills $V(0, f)=V_0$, where $V_0$ is defined in \eqref{V_0-gen}. The parameters satisfy: $\varepsilon_0\in(0,\textnormal{min}\{1,\frac{l}{4}\})$, $\sigma<1$, and $X$ is defined in \eqref{X}.
\end{pro}
{{\begin{proof}
Note that
$$
V(0,f)=\lim_{\varepsilon\rightarrow 0}\frac{\int_{\T}\overline{ I(\varepsilon,f)(w)}{w}\phi'(w)(1-\overline{w}^2)dw}{i\int_{\T} {w}(1-\overline{w}^2)dw}=-\lim_{\varepsilon\rightarrow 0}\frac{\int_{\T}\overline{ I(\varepsilon,f)(w)}{w}(1+\frac{\varepsilon}{G(\varepsilon)}f'(w))(1-\overline{w}^2)dw}{2\pi i},
$$
via the Residue Theorem. We use the decomposition of $I(\varepsilon,f)$ given in \eqref{I_exp}. Let us begin with $I_1(\varepsilon,f)$:
\begin{align*}
I_1(\varepsilon,f)=&I_{1,2}(\varepsilon,f)(w)+I_{1,3}(\varepsilon,f)(w)\\
=&i\frac{\varepsilon}{G(\varepsilon)}\int_{\T}\int_0^1G'\left(\varepsilon\left|(w-\xi)+t\frac{\varepsilon}{G(\varepsilon)}(f(w)-f(\xi))\right|\right)\\
&\times\frac{\textnormal{Re}\left[\left((w-\xi)+t\frac{\varepsilon}{G(\varepsilon)}(f(w)-f(\xi))\right)\overline{(f(w)-f(\xi))}\right]}{|(w-\xi)+t\frac{\varepsilon}{G(\varepsilon)}(f(w)-f(\xi))|}dt\, d\xi\\
&+\frac{i}{G(\varepsilon)}\int_{\T}G(\varepsilon|\phi(w)-\phi(\xi)|)f'(\xi)\, d\xi.
\end{align*}
Using (H4) and the Dominated Convergence Theorem, we get
\begin{equation}\label{I1-0}
I_1(0,f)=i\int_{\T}\frac{G'(|w-\xi|)}{|w-\xi|}\textnormal{Re}\left[(w-\xi)\overline{(f(w)-f(\xi))}\right]\, d\xi+i\int_{\T}G(|w-\xi|)f'(\xi)\, d\xi.
\end{equation}
Note that the Dominated Convergence Theorem can be applied since the limit in (H4) is uniform. We can compute the above integrals in the following way
\begin{align}\label{int-comp}
\int_{\T}\frac{G'(|w-\xi|)}{|w-\xi|}\textnormal{Re}\left[(w-\xi)\overline{(f(w)-f(\xi))}\right]\, d\xi=&\sum_{n\geq 1}a_n\int_{\T}\frac{G'(|w-\xi|)}{|w-\xi|}\textnormal{Re}\left[(w-\xi){(w^n-\xi^n)}\right]\, d\xi\nonumber\\
=&\sum_{n\geq 1}\frac{a_n}{2}\Big\{\overline{w}^n\int_{\T}\frac{G'(|1-\xi|)}{|1-\xi|}\overline{\left[(1-\xi){(1-\xi^n)}\right]}\, d\xi\nonumber\\
&+{w}^{n+2}\int_{\T}\frac{G'(|1-\xi|)}{|1-\xi|}{\left[(1-\xi){(1-\xi^n)}\right]}\, d\xi\Big\},\nonumber\\
\int_{\T}G(|w-\xi|)f'(\xi)\, d\xi=&-\sum_{n\geq 1}a_n n\int_{\T}G(|w-\xi|)\frac{1}{\xi^{n+1}} \, d\xi\nonumber\\
=&-\sum_{n\geq 1}a_n n\overline{w}^n\int_{\T}G(|1-\xi|)\frac{1}{\xi^{n+1}} \, d\xi,
\end{align}
where $f(w)=\sum_{n\geq 1}a_nw^{-n}.$ Note also that
\begin{align*}
\int_{\T}\frac{G'(|1-\xi|)}{|1-\xi|}\overline{\left[(1-\xi){(1-\xi^n)}\right]}\, d\xi,\int_{\T}\frac{G'(|1-\xi|)}{|1-\xi|}{\left[(1-\xi){(1-\xi^n)}\right]}\, d\xi,\int_{\T}G(|1-\xi|)\frac{1}{\xi^{n+1}} \, d\xi\in i\R.
\end{align*}
In this way,
\begin{align*}
\lim_{\varepsilon\rightarrow 0}\int_{\T} \overline{I_1(\varepsilon,f)(w)}w&\left(1+\frac{\varepsilon}{G(\varepsilon)}f'(w)\right)(1-\overline{w}^2)dw=\int_{\T} \overline{I_1(0,f)(w)}w(1-\overline{w}^2)dw\\
=&\sum_{n\geq 1}a_n\Big\{\frac{i}{2}\int_{\T}\frac{G'(|1-\xi|)}{|1-\xi|}\overline{\left[(1-\xi){(1-\xi^n)}\right]}\, d\xi\int_{\T}w^nw(1-\overline{w}^2)dw\\
&+i\int_{\T}\frac{G'(|1-\xi|)}{|1-\xi|}{\left[(1-\xi){(1-\xi^n)}\right]}\, d\xi\int_{\T}\overline{w}^{n+2}w(1-\overline{w}^2)dw\\
&-in\int_{\T}G(|1-\xi|)\frac{1}{\xi^{n+1}} \, d\xi\int_{\T}w^nw(1-\overline{w}^2)dw\Big\}\\
=&0,
\end{align*}
by the Residue Theorem. Let us move on $I_3(\varepsilon,f)$. Here, we use also Taylor formula \eqref{Taylor-formula} for $z_1=-a-kl+ih$ and $z_2=\varepsilon(\phi(w)+\phi(\xi))$, finding
\begin{align}\label{I3-taylor}
I_3(\varepsilon,f)(w)=&\frac{i}{\varepsilon}\sum_{k\in\Z}\int_{\T}G(|-a-kl+ih|)\, d\xi\nonumber\\
&+i\sum_{k\in\Z}\int_{\T}\int_0^1G'\left(\left|-a-kl+ih+\varepsilon t(\phi(w)+\phi(\xi))\right|\right)\nonumber\\
&\times\frac{\textnormal{Re}\left[\left(-a-kl+ih+\varepsilon t(\phi(w)+\phi(\xi))\right)\overline{(\phi(w)+\phi(\xi))}\right]}{|-a-kl+ih+\varepsilon t(\phi(w)+\phi(\xi))|}dt\, d\xi\nonumber\\
&+\frac{i}{G(\varepsilon)}\sum_{k\in\Z}\int_{\T}G(|\varepsilon(\phi(w)+\phi(\xi))-a-kl+ih|)f'(\xi)\, d\xi\nonumber\\
=&i\sum_{k\in\Z}\int_{\T}\int_0^1G'\left(\left|-a-kl+ih+\varepsilon t(\phi(w)+\phi(\xi))\right|\right)\nonumber\\
&\times\frac{\textnormal{Re}\left[\left(-a-kl+ih+\varepsilon t(\phi(w)+\phi(\xi))\right)\overline{(\phi(w)+\phi(\xi))}\right]}{|-a-kl+ih+\varepsilon t(\phi(w)+\phi(\xi))|}dt\, d\xi\nonumber\\
&+\frac{i}{G(\varepsilon)}\sum_{k\in\Z}\int_{\T}G(|\varepsilon(\phi(w)+\phi(\xi))-a-kl+ih|)f'(\xi)\, d\xi\nonumber\\
=:&I_{3,1}(\varepsilon,f)(w)+I_{3,2}(\varepsilon,f)(w).
\end{align}
Let us check that $|z_2|<|z_1|$, which is necessary to use Taylor formula:
$$
|z_2|=\varepsilon|\phi(w)+\phi(\xi)|\leq 2\varepsilon||\phi||_{L^{\infty}}\leq 4\varepsilon<|a+l-ih|\leq |z_1|.
$$
Note that
$$
I_{3,2}(0,f)=0.
$$
For the other term, we achieve using (H4) that
\begin{align*}
I_{3,1}(0,f)(w)=&i\sum_{k\in\Z}\int_{\T}G'\left(\left|-a-kl+ih\right|\right)\frac{\textnormal{Im}\left[\left(-a-kl+ih\right)\overline{(w+\xi)}\right]}{|-a-kl+ih|}\, d\xi\\
=&\frac{1}{2}\sum_{k\in\Z}\int_{\T}G'\left(\left|-a-kl+ih\right|\right)\frac{\left(-a-kl+ih\right)\overline{(w+\xi)}}{|-a-kl+ih|}\, d\xi\\
&-\frac{1}{2}\sum_{k\in\Z}\int_{\T}G'\left(\left|-a-kl+ih\right|\right)\frac{\overline{\left(-a-kl+ih\right)}{(w+\xi)}}{|-a-kl+ih|}\, d\xi\\
=&i\pi\sum_{k\in\Z}G'\left(\left|-a-kl+ih\right|\right)\frac{\left(-a-kl+ih\right)}{|-a-kl+ih|}\\
=&-\pi V_0,
\end{align*}
by the Residue Theorem and the definition of $V_0$ given in \eqref{V_0-gen}. Using the same ideas for $I_2(\varepsilon,f)$, we find that
\begin{align*}
I_2(\varepsilon,f)(w)=&i\sum_{0\neq k\in\Z}\int_{\T}\int_0^1G'\left(\left|-kl+\varepsilon t(\phi(w)+\phi(\xi))\right|\right)\\
&\times\frac{\textnormal{Re}\left[\left(-kl+\varepsilon t(\phi(w)+\phi(\xi))\right)\overline{(\phi(w)+\phi(\xi))}\right]}{|-kl+\varepsilon t(\phi(w)+\phi(\xi))|}dt\, d\xi\\
&+\frac{i}{G(\varepsilon)}\sum_{0\neq k\in\Z}\int_{\T}G(|\varepsilon(\phi(w)+\phi(\xi))-kl|)f'(\xi)\, d\xi,
\end{align*}
and then,
$$
I_2(0,f)=-i\sum_{0\neq k\in\Z}kl\frac{G'(|kl|)}{|kl|}\textnormal{Re}\left[\overline{w-\xi}\right]\, d\xi=0.
$$
This implies that
$$
\lim_{\varepsilon\rightarrow 0}\int_{\T} \overline{I_2(\varepsilon,f)(w)}w\left(1+\frac{\varepsilon}{G(\varepsilon)}f'(w)\right)(1-\overline{w}^2)dw=0.
$$
Finally, we find
\begin{align*}
\lim_{\varepsilon\rightarrow 0}\int_{\T}\overline{ I(\varepsilon,f)(w)}{w}\left(1+\frac{\varepsilon}{G(\varepsilon)}f'(w)\right)(1-\overline{w}^2)dw=&-\frac{1}{\pi}\int_{\T}\overline{ I_{3,1}(0,f)(w)}{w}(1-\overline{w}^2)dw\\
=&V_0\int_{\T}{w}(1-\overline{w}^2)dw\\
=&-2\pi i V_0,
\end{align*}
getting the announced result, i.e.,
$
V(0,f)=V_0.
$
\end{proof}}}

\begin{pro}\label{Prop-reggen}
Let $G$ satisfies the hypothesis (H1)--(H4) of Proposition  \ref{Gen-point}, Proposition \ref{Prop-trivialsol-gen} and Proposition \ref{Prop-trivialsol2Qgen}, and
\begin{enumerate}
\item[(H5)] {$
\frac{d}{d\varepsilon}\frac{\tilde{G}(\varepsilon r)}{\tilde{G}(\varepsilon)\tilde{G}(r)}\rightarrow 0, \,  \frac{d}{d\varepsilon}\frac{\varepsilon\tilde{G}'(\varepsilon r)}{\tilde{G}(\varepsilon)\tilde{G}'(r)}\rightarrow 0, \textnormal{when } \varepsilon\rightarrow 0,
$
uniformly in $r\in(0,2)$.}
\end{enumerate}
If $V$ sets \eqref{function-V-gen}, then $$\tilde{F}:(-\varepsilon_0,\varepsilon_0)\times B_{X_{1-\beta_2}}(0,\sigma)\rightarrow Y_{1-\beta_2},$$ with $\tilde{F}(\varepsilon,f)=F(\varepsilon,f,V(\varepsilon,f))$, is well--defined and $C^1$. The spaces $X$ and $Y$ are defined in \eqref{X}-\eqref{Y} taking $\alpha=1-\beta_2$, and the parameters satisfy $\varepsilon_0\in(0,\textnormal{min}\{1,\frac{l}{4}\})$ and $\sigma<1$.
\end{pro}
\begin{proof}
The proof has three steps: the symmetry of $F$, regularity of $V$, and regularity of $\tilde{F}$.

\medskip
\noindent
{\it $\bullet$ First step: Symmetry of $F$.} 
Let us prove that $F(\varepsilon,f,V)(e^{i\theta})=\sum_{n\geq 1}f_n\sin(n\theta)$ with $f_n\in\R$, i.e., checking that $F$ verifies $F(\varepsilon,f,V)(\overline{w})=-F(\varepsilon,f,V)(w)$.
First, we work with $I(\varepsilon,f)$ showing that $I(\varepsilon,f)(\overline{w})=\overline{I(\varepsilon,f)(w)}:$
\begin{align*}
\overline{I(\varepsilon,f)(w)}=&-\frac{1}{\pi\varepsilon}\sum_{k\in\Z}\overline{\int_{\T} G(|\varepsilon(\phi(w)-\phi(\xi))-kl|)\phi'(\xi)\, d\xi}\\
&-\frac{1}{\pi\varepsilon}\sum_{k\in\Z}\overline{\int_{\T} G(|\varepsilon(\phi(w)+\phi(\xi))-a-kl+ih|)\phi'(\xi)\, d\xi}\\
=&\frac{1}{\pi\varepsilon}\sum_{k\in\Z}{\int_{\T} G(|\varepsilon(\phi(w)-\phi(\overline{\xi}))-kl|)\overline{\phi'(\overline{\xi})}\, d\xi}\\
&+\frac{1}{\pi\varepsilon}\sum_{k\in\Z}{\int_{\T} G(|\varepsilon(\phi(w)+\phi(\overline{\xi}))-a-kl+ih|)\overline{\phi'(\overline{\xi})}\, d\xi}\\
=&-\frac{1}{\pi\varepsilon}\sum_{k\in\Z}{\int_{\T} G(|\varepsilon(\phi(w)-\phi(\overline{\xi}))-kl|){\phi'({\xi})}\, d\xi}\\
&-\frac{1}{\pi\varepsilon}\sum_{k\in\Z}{\int_{\T} G(|\varepsilon(\phi(w)+\phi(\overline{\xi}))-a-kl+ih|){\phi'({\xi})}\, d\xi}\\
=&-\frac{1}{\pi\varepsilon}\sum_{k\in\Z}{\int_{\T} G(|\varepsilon(\phi(\overline{w})-\phi({\xi}))-kl|){\phi'({\xi})}\, d\xi}\\
&-\frac{1}{\pi\varepsilon}\sum_{k\in\Z}{\int_{\T} G(|\varepsilon(\phi(\overline{w})+\phi({\xi}))-a-kl+ih|){\phi'({\xi})}\, d\xi}\\
=&I(\varepsilon,f)(\overline{w}).
\end{align*}
Second, we prove that $V\in\R$ analyzing the denominator and the numerator of its expression:
\begin{align*}
2i\textnormal{Im}\Big[\int_{\T}\overline{ I(\varepsilon,f)(w)}&{w}{\phi'(w)}(1-\overline{w}^2)dw\Big]\\
=&\int_{\T}\overline{ I(\varepsilon,f)(w)}{w}{\phi'(w)}(1-\overline{w}^2)dw-\overline{\int_{\T}\overline{ I(\varepsilon,f)(w)}{w}{\phi'(w)}(1-\overline{w}^2)dw}\\
=&\int_{\T}\overline{ I(\varepsilon,f)(w)}{w}{\phi'(w)}(1-\overline{w}^2)dw+\int_{\T}\overline{ I(\varepsilon,f)(w)}{w}\overline{\phi'(\overline{w})}(1-\overline{w}^2)dw\\
=&\int_{\T}\overline{ I(\varepsilon,f)(w)}{w}{\phi'(w)}(1-\overline{w}^2)dw-\int_{\T}\overline{ I(\varepsilon,f)(w)}{w}{\phi'({w})}(1-\overline{w}^2)dw\\
=&0,
\end{align*}
and
\begin{align*}
2i\textnormal{Im}\left[\int_{\T} {w}{\phi'(w)}(1-\overline{w}^2)dw\right]=&\int_{\T} {w}{\phi'(w)}(1-\overline{w}^2)dw-\overline{\int_{\T} {w}{\phi'(w)}(1-\overline{w}^2)dw}\\
=&\int_{\T} {w}{\phi'(w)}(1-\overline{w}^2)dw+\int_{\T} {w}\overline{\phi'(\overline{w})}(1-\overline{w}^2)dw\\
=&\int_{\T} {w}{\phi'(w)}(1-\overline{w}^2)dw-\int_{\T} {w}{\phi'(w)}(1-\overline{w}^2)dw\\
=&0.
\end{align*}
Thirdly, from the above computations we arrive at
\begin{align*}
F(\varepsilon,f,V)(\overline{w})=&\textnormal{Re}\left[\left\{\overline{I(\varepsilon,f)(\overline{w})}- \overline{V}\right\}\overline{w}{\phi'(\overline{w})}\right]\\
=&-\textnormal{Re}\left[\left\{{I(\varepsilon,f)(w)}-{V}\right\}\overline{w}\overline{\phi'(w)}\right]\\
=&-F(\varepsilon,f,V)(w).
\end{align*}
Now, we have that $F(\varepsilon,f,V)(e^{i\theta})=\sum_{n\geq 1}f_n\sin(n\theta)$. Moreover, the condition \eqref{function-V-gen} agrees with the fact that $f_1=0$, we refer to the first step in the proof of Proposition \ref{Prop-regEuler} for more details.

\medskip
\noindent
{\it $\bullet$ Second step: Regularity of $V$.} 
In similarity with the Euler equations, we need to study first the denominator:
$$
\int_{\T}w\phi'(w)(1-\overline{w}^2)dw=i\int_{\T}w(w+\varepsilon f'(w))(1-\overline{w}^2)dw=2\pi+i\varepsilon\int_{\T}wf'(w)dw=2\pi-i\varepsilon\int_{\T}f(w)dw,
$$
where we have used the Residue Theorem. Then, if $|\varepsilon|<\varepsilon_0$ and $f\in B_{X_{1-\beta_2}}(0,\sigma)$, the denominator is not vanishing. Moreover, it is $C^1$ in $f$ and $\varepsilon\in(-\varepsilon_0,\varepsilon_0)$. By the expression of $V$, it remains to study the regularity of $J(\varepsilon,f)$:
$$
J(\varepsilon,f)(w)=\int_{\T}\overline{ I(\varepsilon,f)(w)}{w}{\phi'(w)}(1-\overline{w}^2)dw.
$$
We use the decomposition of $I(\varepsilon,f)$ given in \eqref{I_exp}.
First, we began with the continuity is both variables. Fixing $f\in B_{X_{1-\beta_2}}(0,\sigma)$, note from Proposition \ref{Prop-trivialsol2Qgen} that
$$
J(0,f)=2\pi V_0,
$$
and then $J$ is continuous in $\varepsilon\in(-\varepsilon_0,\varepsilon_0)$. Now, fixing $\varepsilon\neq 0$, by Lemma \ref{Lem-pottheory} and (H3), we get easily that $I_i(\varepsilon,f)\in C^{1-\beta_2}(\T)$, for $i=1,2,3$.

Secondly, we study the differentiability properties. Fix again $f\in B_{X_{1-\beta_2}}(0,\sigma)$, and we differentiate with respect to $\varepsilon$:
\begin{align}\label{aux1}
\frac{d}{d\varepsilon} J(\varepsilon,f)(w)=&\int_{\T}\overline{ d_{\varepsilon}I(\varepsilon,f)(w)}{w}{\phi'(w)}(1-\overline{w}^2)dw\nonumber\\
&+i\left(\frac{1}{G(\varepsilon)}-\frac{\varepsilon G'(\varepsilon)}{G(\varepsilon)^2}\right)\int_{\T}\overline{ I(\varepsilon,f)(w)}{w}{f'(w)}(1-\overline{w}^2)dw.
\end{align}
The last expression is continuous when $\varepsilon\neq 0$ and now we aim to pass to the limit $\varepsilon\rightarrow 0$. By (H3) and (H4), we have that
$$
\lim_{\varepsilon \rightarrow 0} \left(\frac{1}{G(\varepsilon)}-\frac{\varepsilon G'(\varepsilon)}{G(\varepsilon)^2}\right)=\lim_{\varepsilon \rightarrow 0} \left(\frac{1}{G(\varepsilon)}-\frac{G'(1)}{G(\varepsilon)}\right)=0.
$$
By using Proposition \ref{Prop-trivialsol2Qgen}, we find
\begin{align}\label{I-0}
\lim_{\varepsilon \rightarrow 0}{ I(\varepsilon,f)(w)}=&-\frac{i}{\pi}\int_{\T}\frac{G'(|w-\xi|)}{|w-\xi|}\textnormal{Re}\left[(w-\xi)\overline{(f(w)-f(\xi))}\right]\, d\xi\nonumber\\
&-\frac{i}{\pi}\int_{\T}G(|w-\xi|)f'(\xi)\, d\xi+V_0,
\end{align}
which implies that
$$
\lim_{\varepsilon\rightarrow 0}\left(\frac{1}{G(\varepsilon)}-\frac{\varepsilon G'(\varepsilon)}{G(\varepsilon)^2}\right)\int_{\T}\overline{ I(\varepsilon,f)(w)}{w}{f'(w)}(1-\overline{w}^2)dw=0.
$$

Let us analyze the first term of \eqref{aux1}. Using the decomposition \eqref{I_exp}, we begin with
$$
\int_{\T}\overline{ d_{\varepsilon}I_1(\varepsilon,f)(w)}{w}{\phi'(w)}(1-\overline{w}^2)dw.
$$
We use also the decomposition for $I_1$ given in \eqref{I1-decom}. For the last term, one has that
$$
\frac{d}{d\varepsilon}I_{1,3}(\varepsilon,f)(w)=i\int_{\T}\frac{d}{d\varepsilon}\left(\frac{G(\varepsilon|\phi(w)-\phi(\xi)|)}{G(\varepsilon)}\right)f'(\xi)\, d\xi,
$$
which is continuous in $\varepsilon\neq 0$ and $f\in B_{X_{1-\beta_2}}(0,\sigma)$. Moreover, (H5) implies that
$$
\lim_{\varepsilon\rightarrow 0}\frac{d}{d\varepsilon}I_{1,3}(\varepsilon,f)(w)=0,
$$
for any $w\in\T$. Note that we can use the Dominated Convergence Theorem since the limit in (H5) is uniform. We can differentiate $I_{1,2}$ for $\varepsilon\neq 0$ and using once again (H5), we have that
\begin{align*}
\lim_{\varepsilon\rightarrow 0}\frac{d}{d\varepsilon}I_{1,2}(\varepsilon,f)(w)=0.
\end{align*}
Let us now show the idea of $I_3(\varepsilon,f)$, and $I_2(\varepsilon,f)$ will read similarly. Here, we use Taylor formula \eqref{Taylor-formula} as it was done in \eqref{I3-taylor}:
\begin{align*}
I_3(\varepsilon,f)(w)=&I_{3,1}(\varepsilon,f)(w)+I_{3,2}(\varepsilon,f)(w)\\
=&i\sum_{k\in\Z}\int_{\T}\int_0^1G'\left(\left|-a-kl+ih+\varepsilon t(\phi(w)+\phi(\xi))\right|\right)\\
&\times\frac{\textnormal{Re}\left[\left(-a-kl+ih+\varepsilon t(\phi(w)+\phi(\xi))\right)\overline{(\phi(w)+\phi(\xi))}\right]}{|-a-kl+ih+\varepsilon t(\phi(w)+\phi(\xi))|}dt\, d\xi\\
&+\frac{i}{G(\varepsilon)}\sum_{k\in\Z}\int_{\T}G(|\varepsilon(\phi(w)+\phi(\xi))-a-kl+ih|)f'(\xi)\, d\xi.
\end{align*}
These two expressions are smooth in $\varepsilon$. In fact, we can check that $\frac{d}{d\varepsilon}I_3(\varepsilon,f)$ is continuous in $\varepsilon\in(-\varepsilon_0,\varepsilon_0)$ and $f\in B_{X_{1-\beta_2}}(0,\sigma)$.

Now, fix $\varepsilon\neq 0$ and we focus on the regularity with respect to $f$. The integral $I_1$ is the more delicate one since the kernel is singular. Remark the expression of this term:
\begin{align*}
I_1(\varepsilon,f)=&i\frac{\varepsilon}{G(\varepsilon)}\int_{\T}\int_0^1G'\left(\varepsilon\left|(w-\xi)+t\frac{\varepsilon}{G(\varepsilon)}(f(w)-f(\xi))\right|\right)\\
&\times\frac{\textnormal{Re}\left[\left((w-\xi)+t\frac{\varepsilon}{G(\varepsilon)}(f(w)-f(\xi))\right)\overline{(f(w)-f(\xi))}\right]}{|(w-\xi)+t\frac{\varepsilon}{G(\varepsilon)}(f(w)-f(\xi))|}dt\, d\xi\\
&+\frac{i}{G(\varepsilon)}\int_{\T}G(\varepsilon|\phi(w)-\phi(\xi)|)f'(\xi)\, d\xi.
\end{align*}
Then,
\begin{align*}
\partial_f I_1(\varepsilon,f)h(w)=&i\frac{\varepsilon^3}{G(\varepsilon)^2}\int_{\T}\int_0^1tG''\left(\varepsilon\left|(w-\xi)+t\frac{\varepsilon}{G(\varepsilon)}(f(w)-f(\xi))\right|\right)\\
&\times\frac{\textnormal{Re}\left[\left((w-\xi)+t\frac{\varepsilon}{G(\varepsilon)}(f(w)-f(\xi))\right)\overline{(f(w)-f(\xi))}\right]}{|(w-\xi)+t\frac{\varepsilon}{G(\varepsilon)}(f(w)-f(\xi))|^2}\\
&\times\left\{\left((w-\xi)+\frac{t\varepsilon}{G(\varepsilon)}(f(w)-f(\xi))\right)\cdot (h(w)-h(\xi))\right\}dt\, d\xi\\
&-i\frac{\varepsilon^2}{G(\varepsilon)^2}\int_{\T}\int_0^1tG'\left(\varepsilon\left|(w-\xi)+t\frac{\varepsilon}{G(\varepsilon)}(f(w)-f(\xi))\right|\right)\\
&\times\frac{\textnormal{Re}\left[\left((w-\xi)+t\frac{\varepsilon}{G(\varepsilon)}(f(w)-f(\xi))\right)\overline{(f(w)-f(\xi))}\right]}{|(w-\xi)+t\frac{\varepsilon}{G(\varepsilon)}(f(w)-f(\xi))|^3}dt\, d\xi\\
&\times\left\{\left((w-\xi)+\frac{t\varepsilon}{G(\varepsilon)}(f(w)-f(\xi))\right)\cdot (h(w)-h(\xi))\right\}dt\, d\xi\\
&+i\frac{\varepsilon}{G(\varepsilon)}\int_{\T}\int_0^1G'\left(\varepsilon\left|(w-\xi)+t\frac{\varepsilon}{G(\varepsilon)}(f(w)-f(\xi))\right|\right)\\
&\times\frac{\textnormal{Re}\left[\left((w-\xi)+t\frac{\varepsilon}{G(\varepsilon)}(h(w)-h(\xi))\right)\overline{(f(w)-f(\xi))}\right]}{|(w-\xi)+t\frac{\varepsilon}{G(\varepsilon)}(f(w)-f(\xi))|}dt\, d\xi\\
&+i\frac{\varepsilon}{G(\varepsilon)}\int_{\T}\int_0^1G'\left(\varepsilon\left|(w-\xi)+t\frac{\varepsilon}{G(\varepsilon)}(f(w)-f(\xi))\right|\right)\\
&\times\frac{\textnormal{Re}\left[\left((w-\xi)+t\frac{\varepsilon}{G(\varepsilon)}(f(w)-f(\xi))\right)\overline{(h(w)-h(\xi))}\right]}{|(w-\xi)+t\frac{\varepsilon}{G(\varepsilon)}(f(w)-f(\xi))|}dt\, d\xi\\
&+\frac{i}{G(\varepsilon)}\int_{\T}G(\varepsilon|\phi(w)-\phi(\xi)|)h'(\xi)\, d\xi\\
&+\frac{i\varepsilon^2}{G(\varepsilon)^2}\int_{\T}\frac{G'(\varepsilon|\phi(w)-\phi(\xi)|)}{|\phi(w)-\phi(\xi)|}f'(\xi)\\
&\times \left\{\left((w-\xi)+\frac{\varepsilon}{G(\varepsilon)}(f(w)-f(\xi))\right)\cdot (h(w)-h(\xi))\right\}\, d\xi.
\end{align*}
For any $\varepsilon\neq 0$, the above expression is continuous in $f$ by using Lemma \ref{Lem-pottheory}. Moreover, using (H4) we can obtain the limit when $\varepsilon\rightarrow 0$ as
\begin{align}\label{dfI1}
\partial_f I_1(0,f)h(w)=i\int_{\T}\frac{G'(|w-\xi|)}{|w-\xi|}\textnormal{Re}\left[(w-\xi)\overline{(h(w)-h(\xi))}\right]\, d\xi+i\int_{\T}G(|w-\xi|)h'(\xi)\, d\xi.
\end{align}
Note that it agrees when differentiating with respect to $f$ in \eqref{I1-0}.

For the other two integrals, notice that $I_2$ and $I_3$  are not singular integrals due to $|\varepsilon(\phi(w)-\phi(\xi))-kl|$ it not vanishing for $k\neq 0$ and neither $|\varepsilon(\phi(w)-\phi(\xi))-a-kl+ih|$, for any $k\in\Z$. This gives us that $I_2$ and $I_3$ are $C^1$. Let us show the idea of $I_2$:
\begin{align*}
\partial_f I_2(\varepsilon,f)h(w)=&\partial_f \frac{1}{\varepsilon}\sum_{0\neq k\in\Z}\int_{\T} G(|\varepsilon(\phi(w)-\phi(\xi))-kl|)\phi'(\xi)\, d\xi\\
=&\frac{i}{\varepsilon}\frac{\varepsilon^2}{G(\varepsilon)}\sum_{0\neq k\in\Z}\int_{\T} G'(|\varepsilon(\phi(w)-\phi(\xi))-kl|)\\
&\times\frac{(\varepsilon(\phi(w)-\phi(\xi))-kl)\cdot (h(w)-h(\xi))}{|\varepsilon(\phi(w)-\phi(\xi))-kl|}\phi'(\xi)\, d\xi\\
&+\frac{i}{\varepsilon}\frac{\varepsilon}{G(\varepsilon)}\sum_{0\neq k\in\Z}\int_{\T} G(|\varepsilon(\phi(w)-\phi(\xi))-kl|)h'(\xi)\, d\xi.
\end{align*}
By Lemma \ref{Lem-pottheory}, the last expression is continuous in $f$. Moreover, it does when $\varepsilon\neq 0$ and it is easy to check, with the help of the Convergence Dominated Theorem, that
\begin{equation}\label{dfI2}
\partial_f I_2(0,f)h(w)=0.
\end{equation}
In the same way, one can check that $\partial_f I_3(\varepsilon,f)$ is continuous in both variables and
\begin{equation}\label{dfI3}
\partial_f I_3(0,f)h(w)=0.
\end{equation}

\medskip
\noindent
{\it $\bullet$ Third step: Regularity of $\tilde{F}$.} Since $V(\varepsilon,f)$ is $C^1$ in both variables and using the computations above concerning $I(\varepsilon,f)$, one can easily check that $\tilde{F}$ is $C^1$.
\end{proof}

\subsection{Main result}
Finally, we can announce the result concerning the desingularization of the point model \eqref{PointVortex-gen} in the general system. We need to impose an extra condition to $G$ in order to obtain that the linearized operator is an isomorphism.
\begin{theo}\label{Th-gen}
Consider $G$ satisfying (H1)--(H5) of Proposition \eqref{Gen-point}, Proposition \eqref{Prop-trivialsol-gen} and Proposition \eqref{Prop-trivialsol2Qgen}, and
\begin{enumerate}
\item[(H6)]
$
0\notin \left\{n\int_{\T}G(|1-\xi|)(1-\overline{\xi}^{n+1})\, d\xi-i\int_{\T}\frac{G'(|1-\xi|)}{|1-\xi|}\textnormal{Im}\left[(1-{\xi})(1-{\xi}^n)\right]\, d\xi,\quad n\geq 1\right\}.
$
\end{enumerate}
Let $h, l\in\R$, with $h\neq 0$ and $l>0$, and $a=0$ or $a=\frac{l}{2}$. Then, there exist $D^{\varepsilon}$ such that
\begin{equation}\label{omega_epsilon2-gen}
q_{0,\varepsilon}(x)=\frac{1}{\pi\varepsilon^2}\sum_{k\in\Z}{\bf 1}_{\varepsilon D^{\varepsilon}+kl}(x)-\frac{1}{\pi\varepsilon^2}\sum_{k\in\Z}{\bf 1}_{-\varepsilon D^{\varepsilon}+a-ih+kl}(x),
\end{equation}
defines a horizontal translating solution of \eqref{Generaleq}, with constant  speed, for any $\varepsilon\in(0,\varepsilon_0)$ and small enough $\varepsilon_0>0$. Moreover, $D^{\varepsilon}$ is at least $C^1$.
\end{theo}
{{\begin{proof}
Let us consider $\tilde{F}:(-\varepsilon_0,\varepsilon_0)\times B_{X_{1-\beta_2}}(0,\sigma)\rightarrow Y_{1-\beta_2}$, with $\varepsilon\in(0,1)$, $\varepsilon<\frac{l}{4}$ and $\sigma<1$, defined in Proposition \ref{Prop-reggen}. By that proposition, it is $C^1$ in both variables. Moreover, Proposition \ref{Prop-trivialsol-gen} and Proposition and \ref{Prop-trivialsol2Qgen}, give us that $\tilde{F}(0,0)=0$. In order to implement the Implicit Function Theorem, let us compute the linearized operator:
\begin{align*}
\partial_f \tilde{F}(0,0)h(w)=&\lim_{\varepsilon\rightarrow 0}\textnormal{Re}\left[\left\{\partial_f \overline{I(0,0)}h(w)-\partial_f V(0, 0)h(w)\right\}iw\right.\\
&\left.+\left\{\overline{I(0,0)(w)}-V_0\right\}iw\frac{\varepsilon}{G(\varepsilon)} h'(w)\right]+\frac{i\int_{\T}G(\varepsilon|w-\xi|)\, d\xi}{\pi G(\varepsilon)}\textnormal{Im}\left[h'(w)\right].
\end{align*}
By Proposition \ref{Prop-trivialsol2Qgen}, then $\partial_f V(0, f)h(w)\equiv 0$. In virtue of \eqref{I-0}, we have
$$
{I(0,f)(w)}=-\frac{i}{\pi}\int_{\T}\frac{G'(|w-\xi|)}{|w-\xi|}\textnormal{Re}\left[(w-\xi)\overline{(f(w)-f(\xi))}\right]\, d\xi-\frac{i}{\pi}\int_{\T}G(|w-\xi|)f'(\xi)\, d\xi+V_0,
$$
which implies
$$
I(0,0)(w)=V_0.
$$
Then, we have
\begin{align*}
\partial_f \tilde{F}(0,0)h(w)=&\textnormal{Re}\left[iw\partial_f \overline{I(0,0)}h(w)-\frac{1}{\pi}\overline{\int_{\T}G(|1-\xi|)\, d\xi} h'(w)\right].
\end{align*}
On the other hand, using \eqref{dfI1}-\eqref{dfI2}-\eqref{dfI3} we obtain
\begin{align*}
-\pi\partial_f I(0,0)h(w)=&\partial_f I_1(0,0)h(w)\\
=&i\int_{\T}\frac{G'(|w-\xi|)}{|w-\xi|}\textnormal{Re}\left[(w-\xi)\overline{(h(w)-h(\xi))}\right]\, d\xi+i\int_{\T}G(|w-\xi|)h'(\xi)\, d\xi,
\end{align*}
which amounts to
\begin{align*}
\partial_f \tilde{F}(0,0)h(w)=&\textnormal{Re}\Big[-\frac{w}{\pi}\overline{\int_{\T}\frac{G'(|w-\xi|)}{|w-\xi|}\textnormal{Re}\left[(w-\xi)\overline{(h(w)-h(\xi))}\right]\, d\xi}\\
&-\frac{w}{\pi}\overline{\int_{\T}G(|w-\xi|)h'(\xi)\, d\xi}-\frac{h'(w)}{\pi}\overline{\int_{\T}G(|1-\xi|)\, d\xi} \Big].
\end{align*}
Note that
$
\mathcal{K}:C^{2-\beta_2}(\T)\rightarrow C^{1-\beta_2}(\T),
$
defined by
$$
\mathcal{K}(h)(w)={\int_{\T}\frac{G'(|w-\xi|)}{|w-\xi|}\textnormal{Re}\left[(w-\xi)\overline{(h(w)-h(\xi))}\right]\, d\xi}+{\int_{\T}G(|w-\xi|)h'(\xi)\, d\xi},
$$
is a compact operator since it is smoothing. Since $h\in C^{2-\beta_2}(\T)\mapsto h'\in C^{1-\beta_2}(\T)$ is a Fredholm operator of zero index, then $\partial_f \tilde{F}(0,0)$ so is. As a consequence,  checking that $\partial_f \tilde{F}(0,0)$ is an isomorphism, it is enough to check that the kernel is trivial.  We can compute the integrals involving in the linearized operator using \eqref{int-comp}, and finding
\begin{align*}
\int_{\T}\frac{G'(|w-\xi|)}{|w-\xi|}\textnormal{Re}\left[(w-\xi)\overline{(h(w)-h(\xi))}\right]\, d\xi=&\sum_{n\geq 1}\frac{a_n}{2}\Big\{\overline{w}^n\int_{\T}\frac{G'(|1-\xi|)}{|1-\xi|}\overline{\left[(1-\xi){(1-\xi^n)}\right]}\, d\xi\\
&+{w}^{n+2}\int_{\T}\frac{G'(|1-\xi|)}{|1-\xi|}{\left[(1-\xi){(1-\xi^n)}\right]}\, d\xi\Big\},\\
\int_{\T}G(|w-\xi|)h'(\xi)\, d\xi=&-\sum_{n\geq 1}a_n n\overline{w}^n\int_{\T}G(|1-\xi|)\frac{1}{\xi^{n+1}} \, d\xi,
\end{align*}
where $h(w)=\sum_{n\geq 1}a_nw^{-n}.$ Note also that
\begin{align*}
\int_{\T}G(\varepsilon|1-\xi|)\, d\xi,&\int_{\T}\frac{G'(|1-\xi|)}{|1-\xi|}\overline{\left[(1-\xi){(1-\xi^n)}\right]}\, d\xi,\\
&\int_{\T}\frac{G'(|1-\xi|)}{|1-\xi|}{\left[(1-\xi){(1-\xi^n)}\right]}\, d\xi,\int_{\T}G(|1-\xi|)\frac{1}{\xi^{n+1}} \, d\xi\in i\R.
\end{align*}
Finally, we achieve
\begin{align*}
\partial_f \tilde{F}(0,0)h(w)=&\sum_{n\geq 1}\frac{a_n}{\pi}\textnormal{Re}\Big[\frac{w^{n+1}}{2}\int_{\T}\frac{G'(|1-\xi|)}{|1-\xi|}\overline{\left[(1-\xi){(1-\xi^n)}\right]}\, d\xi\\
&+\frac{\overline{w}^{n+1}}{2}\int_{\T}\frac{G'(|1-\xi|)}{|1-\xi|}{\left[(1-\xi){(1-\xi^n)}\right]}\, d\xi\\
&-nw^{n+1}\int_{\T}G(|1-\xi|)\frac{1}{\xi^{n+1}} \, d\xi-n\overline{w}^{n+1}\int_{\T}G(\varepsilon|1-\xi|)\, d\xi\Big]\\
=&\sum_{n\geq 1}\frac{a_n i}{\pi}\sin((n+1)\theta)\Big\{\frac12\int_{\T}\frac{G'(|1-\xi|)}{|1-\xi|}\overline{\left[(1-\xi){(1-\xi^n)}\right]}\, d\xi\\
&-\frac12\int_{\T}\frac{G'(|1-\xi|)}{|1-\xi|}{\left[(1-\xi){(1-\xi^n)}\right]}\, d\xi\\
&-n\int_{\T}G(|1-\xi|)\frac{1}{\xi^{n+1}} \, d\xi+n\int_{\T}G(\varepsilon|1-\xi|)\, d\xi\Big\}\\
=&\sum_{n\geq 1}\frac{a_n i}{\pi}\sin((n+1)\theta)\Big\{n\int_{\T}G(|1-\xi|)(1-\overline{\xi}^{n+1})\, d\xi\\
&-i\int_{\T}\frac{G'(|1-\xi|)}{|1-\xi|}\textnormal{Im}\left[(1-{\xi})(1-{\xi}^n)\right]\, d\xi\Big\}.
\end{align*}
By (H6), we get that the kernel is trivial and then the linerized operator is an isomorphism.
\end{proof}}}
As a consequence, we get the result for the generalized surface quasi--geostrophic equation, meaning $G=\frac{C_{\beta}}{2\pi}\frac{1}{|\cdot|^\beta}$, for $\beta\in(0,1)$ and $C_{\beta}=\frac{\Gamma\left(\frac{\beta}{2}\right)}{2^{1-\beta}\Gamma\left(\frac{2-\beta}{2}\right)}$. We just have to check that (H6) is verified and these computations are done in \cite{HmidiMateu-pairs} for the vortex pairs.

\begin{theo}\label{Th-gSQG}
Let $h, l\in\R$, with $h\neq 0$ and $l>0$, and $a=0$ or $a=\frac{l}{2}$. Then, there exists $D^{\varepsilon}$ such that
\begin{equation}\label{omega_epsilon2-qsqg}
q_{0,\varepsilon}(x)=\frac{1}{\pi\varepsilon^2}\sum_{k\in\Z}{\bf 1}_{\varepsilon D^{\varepsilon}+kl}(x)-\frac{1}{\pi\varepsilon^2}\sum_{k\in\Z}{\bf 1}_{-\varepsilon D^{\varepsilon}+a-ih+kl}(x),
\end{equation}
defines a horizontal translating solution of the generalized surface quasi-geostrophic equations for $\beta\in(0,1)$, with constant velocity speed, for any $\varepsilon\in(0,\varepsilon_0)$ and small enough $\varepsilon_0>0$. Moreover, $D^{\varepsilon}$ is at least $C^1$.
\end{theo}

\appendix

\section{Special functions}\label{Ap-specialfunctions}
This section recalls some definitions and properties about the Bessel functions. Let us define the Bessel function of the first kind and order $\nu$ by the expansion
$$
J_{\nu}(z)=\sum_{k=0}^{+\infty}\frac{(-1)^k\left(\frac{z}{2}\right)^{\nu+2k}}{k!\, \Gamma(\nu+k+1)},\quad |\textnormal{arg}(z)|<\pi.
$$
In addition, it is known that Bessel functions admit the following integral representation
$$
J_n(z)=\frac{1}{\pi}\int_0^{\pi}\cos(n\theta-z\sin \theta)d\theta,\quad z\in\C,
$$
for $\nu=n\in\Z$. On the other hand, the Bessel functions of imaginary argument, denoted by $I_{\nu}$ and $K_{\nu}$ are defined by
$$
I_{\nu}(z)=\sum_{k=0}^{+\infty}\frac{\left(\frac{z}{2}\right)^{\nu+2k}}{k!\, \Gamma(\nu+k+1)},\quad |\textnormal{arg}(z)|<\pi,
$$
and
$$
K_{\nu}(z)=\frac{\pi}{2}\frac{I_{-\nu}-I_{\nu}(z)}{\sin(\nu\pi)} \quad \nu\in\C\setminus\Z, \, |\textnormal{arg}(z)|<\pi.
$$
An useful expansion for $K_n$ can be found in \cite{Watson}:
\begin{align*}
K_n(z)=&(-1)^{n+1}\sum_{k=0}^{+\infty}\frac{\left(\frac{z}{2}\right)^{n+2k}}{k!(n+k)!}\left(\ln\left(\frac{z}{2}\right)-\frac{1}{2}\varphi(k+1)-\frac12\varphi(n+k+1)\right)\\
&+\frac12\sum_{k=0}^{n-1}\frac{(-1)^k(n-k-1)!}{k!\left(\frac{z}{2}\right)^{n-2k}},
\end{align*}
for $n\in\N^{\star}$.
In the case that we concern, $K_0$, it reads as
\begin{equation}\label{K0-expansion}
K_0(z)=-\ln\left(\frac{z}{2}\right)I_0(z)+\sum_{k=0}^\infty\frac{\left(\frac{z}{2}\right)^{2k}}{(k!)^2}\varphi(k+1),
\end{equation}
where
$$
\varphi(1)=-\gamma \quad \text{ and }\quad \varphi(k+1)=\sum_{m=1}^k \frac{1}{m}-\gamma, \, k\in\N^*.
$$
The constant $\gamma$ is the Euler's constant. Moreover, the following asymtotic behaviour at infinity for $K_0$ is given in \cite{Abramowitz}:
\begin{equation}\label{K0-inf}
K_0(z)\sim \sqrt{\frac{\pi}{2z}}e^{-z},\quad |\textnormal{arg}(z)|<\frac{3}{2}\pi.
\end{equation}
Finally, the derivative of $K_n$ can be expressed by terms of Bessel functions. In particular, $K_0'=-K_1$.

\section{Complex integrals and potential theory}\label{Ap-potentialtheory}
This section concerns some complex integrals, in particular,  the Stokes Theorem, the Cauchy--Pompeiu's formula and some results on singular integrals.

The complex version of the Stokes Theorem reads as follows.  Let  $D$  be a simply connected domain and $f$ a $C^1$ scalar function, then
\begin{equation}\label{Stokes}
\int_{\partial D}f(\xi)\, \, d\xi=2i\int_{D} \partial_{\overline{z}} f(z)\, dA(z),
\end{equation}
where $\partial_{\overline{z}}$ can be identify to the gradient operator in the same way
$$
\nabla=2\partial_{\overline{z}}, \quad \partial_{\overline{z}}\varphi(z):=\frac12\left(\partial_{1}\varphi(z)+i\partial_{2} \varphi(z)\right).
$$
Let us introduce now the Cauchy--Pompeiu's formula. Consider $\varphi:\overline{D}\to\C$  be a ${C}^1$ complex function, then
\begin{equation}\label{Cauchy-Pom}
-\frac{1}{\pi}\int_D \frac{\partial_{\overline{y}}\varphi(y)}{w-y}dA(y)=\frac{1}{2\pi i}\int_{\partial D}\frac{\varphi(w)-\varphi(\xi)}{w-\xi}d\xi.
\end{equation}

Finally, we deal with singular integrals of the type
\begin{equation}\label{operator-T}
\mathcal{T}(f)(w)=\int_{\T} K(w,\xi)f(\xi)\, \, d\xi,\quad w\in\T,
\end{equation}
where $K:\T\times \T\rightarrow\C$ being smooth off the diagonal. The next result focuses on the smoothness of the last operator, whose proof can be found in \cite{Hassainia-Hmidi}. See also \cite{Helms, Kress, LiebLoss}.

\begin{lem}\label{Lem-pottheory}
Let $0\leq \alpha<1$ and consider $K:\T\times\T\rightarrow \C$ with the following properties. There exists $C_0>0$ such that
\begin{itemize}
\item[(i)] $K$ is measurable on $\T\times\T\setminus\{(w,w), w\in\T\}$ and 
$$
|K(w,\xi)|\leq \frac{C_0}{|w-\xi|^\alpha}, \quad \forall w\neq \xi\in\T.
$$
\item[(ii)] For each $\xi\in\T$, $w\mapsto K(w,\xi)$ is differentiable in $\T\setminus\{\xi\}$ and 
$$
|\partial_w K(w,\xi)|\leq \frac{C_0}{|w-\xi|^{1+\alpha}}, \quad \forall w\neq \xi\in\T.
$$
\end{itemize}
Then,
\begin{enumerate}
\item The operator $\mathcal{T}$ defined by \eqref{operator-T} is continuous from $L^{\infty}(\T)$ to $C^{1-\alpha}(\T)$. More precisely, there exists a constant $C_{\alpha}$ depending only on $\alpha$ such that
$$
\|\mathcal{T}(f)\|_{1-\alpha}\leq C_{\alpha}C_0\|f\|_{L^{\infty}}.
$$
\item For $\alpha=0$, the operator $\mathcal{T}$ is continuous from $L^{\infty}(\T)$ to $C^{\beta}(\T)$, for any $0<\beta<1$. That is, there exists a constant $C_{\beta}$ depending only on $\beta$ such that
$$
\|\mathcal{T}(f)\|_{\beta}\leq C_{\beta}C_0\|f\|_{L^{\infty}}.
$$
\end{enumerate}
\end{lem}


\begin{thebibliography}{99}


\bibitem{Abramowitz}{M. Abramowitz, I. A. Stegun, } {\it Handbook of Mathematical Functions: With Formulas, Graphs, and Mathematical Tables.}  Courier Corporation, 1965.

\bibitem{Aref}{H. Aref, } {\it On the equilibrium and stability of a row of point vortices.} J. Fluid Mech. {\bf 290} (1995), 167--181.

\bibitem{Beichman-Denisov}{J. Beichman, S. Denisov, } {\it 2D Euler equation on the strip: Stability of a rectangular patch.} Comm. Part. Differ. Equat., {\bf 42} (2017), 100--120.

\bibitem{B-C} A. L. Bertozzi, P. Constantin, {\it Global regularity for vortex patches.} Comm. Math. Phys.
{\bf 152}(1) (1993), 19--28.

\bibitem{Burbea} { J. Burbea,} {\it Motions of vortex patches.} Lett. Math. Phys. {\bf 6} (1982), 1--16.

 \bibitem{Cas0-Cor0-Gom} {A. Castro, D. C\'ordoba, J. G\'omez-Serrano, }{\it Existence and regularity of rotating global solutions for the generalized surface quasi-geostrophic equations.} Duke Math. J. {\bf 165}(5) (2016), 935--984.
 
\bibitem{Cas-Cor-Gom} {A. Castro, D. C\'ordoba, J. G\'omez-Serrano, }{\it  Uniformly rotating analytic global patch solutions for active scalars}. J. Ann. PDE {\bf 2}(1) (2016),  Art. 1, 34.

\bibitem{CastroCordobaGomezSerrano} {A. Castro, D. C\'ordoba, J. G\'omez-Serrano, } {{\it Uniformly rotating smooth solutions for the incompressible 2D Euler equations,}} arXiv:1612.08964, 2016.


\bibitem{Constantin-Majda-Tabak}{P. Constantin, A. J. Majda, E. Tabak, }{\it Formation of strong fronts in the 2-D quasigeostrophic thermal
active scalar.} Nonlinearity {\bf 7}(6) (1994), 1495--1533.

\bibitem{Cordoba}{D. C\'ordoba, M.A. Fontelos, A.M. Mancho, J.L. Rodrigo, }{\it Evidence of singularities for a family of
contour dynamics equations.} Proc. Natl. Acad. Sci. USA {\bf 102}(17) (2005), 5949--5952.

\bibitem{Chemin} {J.-Y. Chemin,} {{\it Persistance de structures g\'eometriques dans les  fluides incompressibles bidimensionnels. }} Ann. Sci. Ec. Norm. Sup. {\bf 26} (1993), 1--26.

 \bibitem{D-H-R} {D. G. Dritschel, T. Hmidi, C. Renault, } {\it
Imperfect bifurcation for the quasi-geostrophic shallow-water equations}. Arch. Ration. Mech. Anal {\bf 231} (2019), 1853--1915. 

\bibitem{DelaHozHmidiMateuVerdera} {F. De la Hoz, T. Hmidi, J. Mateu, J. Verdera, } {\it Doubly connected V-states for the planar Euler equations. } SIAM J. Math. Anal. {\bf 48} (2016), 1892--1928.

\bibitem{Deem-Zabusky} {G. S. Deem, N. J. Zabusky, } {\it Vortex waves: Stationary ``V-states'', Interactions, Recurrence, and Breaking.} Phys. Rev. Lett. {\bf 40} (1978), 859--862.

\bibitem{GHS}{C. Garc\'ia, T. Hmidi, J. Soler, }{\it Non uniform rotating vortices and periodic orbits for the two--dimensional Euler equations. } arXiv:1807.10017, 2018.

\bibitem{VladimirGryanikBorthOlbers}{V. Gryanik, H. Borth, D. Olbers, }{\it The theory of quasi-geostrophic von K\'arm\'an vortex streets in two-layer fluids on a beta-plane. } J. Fluid Mech. {\bf 505 }(2004), 23--57.

\bibitem{Hassainia-Hmidi}{Z. Hassainia, T. Hmidi, }{\it On the V-States for the generalized quasi-geostrophic equations. }Comm. Math. Phys. {\bf 337}(1) (2015), 321--377.

\bibitem{HMW} Z. Hassainia, N. Masmoudi, M. H. Wheeler, {\it Global bifurcation of rotating vortex patches},  	arXiv:1712.03085, 2017.

\bibitem{Helms}{L. L. Helms, }{\it Potential theory. } Springer-Verlag London, 2009.

\bibitem{HmidiMateu-pairs} T. Hmidi, J. Mateu,  {\it Existence of corotating and counter-rotating vortex pairs for active
scalar equations.} Comm. Math. Phys. {\bf 350}(2) (2017), 699--747.

\bibitem{HmidiMateuVerdera} {T. Hmidi, J. Mateu, J. Verdera, } {\it Boundary regularity of rotating vortex patches. } Arch. Ration. Mech. Anal {\bf 209} (2013), 171--208. 

\bibitem{Karman1}{T. von K\'arm\'an, } {\it \"Uber den Mechanismus des Widerstands, den ein bewegter Korper in einer Fl\"ussigkeit erf\"ahrt.} G\"ottinger Nachrichten, Math. Phys. Kl., (1911), 509--517.

\bibitem{Karman2}{T. von K\'arm\'an, } {\it \"Uber den Mechanismus des Widerstands, den ein bewegter Korper in einer Fl\"ussigkeit erf\"ahrt.} G\"ottinger Nachrichten, Math. Phys. Kl., (1912), 547--556.

\bibitem{Kirchhoff}{G. R. Kirchhoff, }{\it  Vorlesungenber mathematische Physik. Mechanik.} Teubner, Leipzig, 1876.

\bibitem{Kress}{R. Kress, } {\it Linear Integral Equations. } Springer--New York--Heidelberg Dordrecht--London, 2014.

\bibitem{Lamb}{Sir Horace Lamb, } {\it Hydrodynamics. } Cambridge U.P., 1932.


\bibitem{LiebLoss}{E. Lieb, M. Loss, } {\it Analysis. } American Mathematical Society, 1997.

\bibitem{MajdaBertozzi}{A. Majda, A. Bertozzi, }{\it Vorticity and Incompressible Flow.} Cambridge University Press, 2002.

\bibitem{Newton}{P. K. Newton, }{\it The N-Vortex Problem.} Analytical Techniques, Springer, New York, 2001.

\bibitem{Pierrehumbert}{R. T. Pierrehumbert, }{\it A family of steady, translating vortex pairs with distributed vorticity.} J. Fluid Mech., {\bf 99} (1980), 129--144.


\bibitem{Plotka-Dristchel}{H. Plotka, D. G. Dritschel,}{\it Quasi-geostrophic shallow-water vortex-patch equilibria and their stability.} Geophys.
Astrophys. Fluid Dyn. {\bf 106}(6) (2012), 574--595. 

\bibitem{Polvani}{L. M. Polvani, }{\it Geostrophic vortex dynamics.} PhD thesis, MIT/WHOI WHOI-88-48, 1988.

\bibitem{Polvani-Zabusky-Flierl}{ L. M. Polvani, N. J. Zabusky, G. R. Flierl, }{\it Two-layer geostrophic vortex dynamics. Part 1. Upper-layer V-states
and merger.} J. Fluid Mech. {\bf 205} (1989), 215--242.

\bibitem{Rayleigh}{L. Rayleigh, } {\it Acoustical Obervation II.} Philosophical Magazine, {\bf 7}(1879), 149--162.

\bibitem{Rosenhead}{L. Rosenhead, }{\it Double Row of Vortices with Arbitrary Stagger. } Math. Proc. Camb. Philos. Soc., {\bf 25}(2) (1929), 132--138. 

\bibitem{Saffman-Schatzman}{P. G. Saffman, J. C. Schatzman, } {\it Properties of a vortex street of finite vortices.} SIAM J. Sci. Stat. Comp. {\bf 2} (1981), 285--295.

\bibitem{Saffman-Schatzman2}{P. G. Saffman, J. C. Schatzman, } {\it Stability of the K\'arm\'an vortex street. } J . Fluid Mech. {\bf 117} (1982), 171--185

\bibitem{Schatzman}{J. C. Schatzman, } {\it A model for the von K\'arm\'an vortex street. } Ph.D. thesis, California Institute of Technology, 1981.

\bibitem{Strouhal}{V. Strouhal, } {\it Uber eine besondere Art der Tonnerregung.} Annalen der Physik und Chemie, {\bf 5}(1878), 216--251.

\bibitem{Vallis}{G. K. Vallis, }{\it Atmospheric and Oceanic Fluid Dynamics. } Cambridge University Press, 2008.

\bibitem{Yudovich}{Y. Yudovich, }{\it Nonstationary flow of an ideal incompressible liquid. } Zh. Vych. Mat. {\bf 3} (1963), 1032--1066.

\bibitem{Watson}{G. N. Watson. }{\it A Treatise on the Theory of Bessel Functions. } Cambrige University Press , 1944.
\end{thebibliography}
\end{document}